\documentclass[a4paper,leqno,english,10pt]{article}
\usepackage[totalheight=26 true cm, totalwidth=17 true cm,]{geometry}

\usepackage[font=small, margin=1cm]{caption}
\usepackage{color}
\usepackage{xcolor}
\usepackage{amsthm,amsmath,amssymb,amsfonts,latexsym,url}
\usepackage{bm,enumitem}
\usepackage{thmtools}
\usepackage{thm-restate}
\usepackage[colorlinks=false]{hyperref}
\usepackage{soul}
\usepackage{pifont, eucal}

\usepackage{quoting}
\quotingsetup{vskip=0pt}

\usepackage[utf8]{inputenc}
\usepackage[english]{babel}
\usepackage{hyperref}
\usepackage{cleveref}

\newtheorem{thm}{Theorem}[section]
\newtheorem{theorem}[thm]{Theorem}
\newtheorem{cor}[thm]{Corollary}
\newtheorem{lemma}[thm]{Lemma}
\newtheorem{prop}[thm]{Proposition}

\newtheorem{step}{Step}

\theoremstyle{definition}

\newtheorem{notation}[thm]{Notation}

\theoremstyle{remark}
\newtheorem{rmk}[thm]{Remark}
\newtheorem{remark}[thm]{Remark}
\newtheorem{exa}[thm]{Example}

\numberwithin{equation}{section}

\def\$$endproof{\eqno{\qedhere}$$\end{proof}}

\def\cR{\mathcal{R}}

\def\cDT{{\mathcal{D}\mathcal{T}}}
\def\cF{\mathcal{F}}

\def\bF{\mathbf{F}}
\def\bK{\mathbf{K}}

\def\C{\mathbb{C}}
\def\F{\mathbb{F}}
\def\K{\mathbb{K}}
\def\Q{\mathbb{Q}}
     
\def\Z{\mathbb{Z}}

\def\Kalg{\widetilde{\K}} 
\def\Falg{\widetilde{\F}}

\usepackage{tikz}
\usetikzlibrary{trees}
\tikzset{mytree/.style={fill,circle,inner sep=-.5pt, level
    distance=5mm,
    level 1/.style={sibling distance=6mm,level distance=5mm},
    level 2/.style={sibling distance=2mm,level distance=6mm},
    level 3/.style={sibling distance=1mm}}}

\usetikzlibrary{arrows.meta, positioning}

\usetikzlibrary{lindenmayersystems}
\def\trianglewidth{8cm}%
\pgfdeclarelindenmayersystem{Sierpinski triangle}{
    \symbol{X}{\pgflsystemdrawforward}
    \symbol{Y}{\pgflsystemdrawforward}
    \rule{X -> X-Y+X+Y-X}
    \rule{Y -> YY}
}%

    \tikzset{d/.style={minimum width=1pt,inner sep=1pt,circle,fill=black}}
    \tikzset{f/.style={minimum width=1.2pt,inner sep=1pt,circle,fill=red}}
    \tikzset{g/.style={circle,fill=white}}

\usepackage{xcolor}

\newcommand{\cS}{\mathcal{S}}

\def\l{\left}
\def\r{\right}

\def\GL{{\rm GL}}

\usepackage[textwidth=1.6cm]{todonotes}

\def\change#1{#1}

\def\phir{\Phi_R}

\title{Inhomogeneous order $1$ iterative functional equations with applications to combinatorics}

\author{Lucia Di Vizio,  Gwladys Fernandes, and  Marni Mishna
\footnote{This project has received funding from the ANR project
\href{https://specfun.inria.fr/chyzak/DeRerumNatura/}{DeRerumNatura,
ANR-19-CE40-0018}. Marni Mishna was supported by NSERC Discovery Grant RGPIN-2017-04157.
Gwladys Fernandes had a Hadamard Lecturer position, supported by the Fondation Mathématique Jacques Hadamard (FMJH).
We are grateful to CNRS IRL-PIMS International for their support of this collaboration.}}

\begin{document}
\bibliographystyle{alpha}

\maketitle

\begin{abstract}
We show that if a Laurent series $f\in\C((t))$ satisfies a particular kind of linear iterative equation, then $f$ is either a rational function or it is differentially transcendental over $\C(t)$.  This condition is more precisely stated as follows:
We consider $R,b\in \C(t)$ with
$R(0)=0$, such that $f(R(t))=f(t)+b(t)$.
If either $R'(0)=0$ or $R'(0)$ is a root of unity,
then either $f$ is a rational function, or $f$ does not satisfy a polynomial differential equation.
More generally a solution of a functional equation of the form $f(R(t))=a(t)f(t)+b(t)$ will be either differentially trascendental or the solution of an
inhomogeneous linear differential equation of order $1$ with rational coefficients.
\par
We illustrate how to apply these results to deduce the differential transcendence of combinatorial generating functions by considering three examples: the ordinary generating function for a family of
complete trees; the Green function for excursions on the Sierpinski graph; and a series related to the enumeration of permutations avoiding the consecutive pattern 1423.
\par
The proof strategy is inspired by the Galois theory of functional
equations and relies on the property of the dynamics of $R$.

\medskip

\end{abstract}

\setcounter{tocdepth}{3}
\tableofcontents

\section{Introduction}

In the study of discrete structures, enumerative data can be encoded
in the coefficients of a formal power series, its generating function.
It can be both theoretically, and practically useful to know
which functional equations such a generating function satisfies and which ones it doesn't. This is true in particular with reference to the differential properties of the generating function. Recall that
a series is \emph{differentially transcendental} over the field $\C(t)$ of rational functions with complex coefficients,
if for any integer $n\geq 0$, and for any polynomial $P$ in $n+1$ variables with coefficients in $\C(t)$,
we have $P(f,f',\dots,f^{(n)})\neq 0$, where $f^{(n)}$ is the $n$-fold derivative
of $f$ with respect to $t$, i.e. if it satisfies no algebraic differential equation with coefficients in $\C(t)$.
\par
In this work we establish the differential transcendence of series arising from the study of three different types of combinatorial objects: families of trees under a constraint on the leaves; walks on fractal-like graphs that start and end at the same point; and a family of consecutive-pattern avoiding permutations. These three examples all have something in common: series that satisfy an order 1 \emph{iterative} functional equation of the form $f(R(t))=a(t)f(t)+b(t)$ with $a, b$ and $R$ rational functions. These equations are ultimately the object studied in this paper.
\par
We have made the choice to
present the applications first, which allows us to state our main theorems through the
lens of enumerative combinatorics: Indeed the statements are straightforward to apply
and unify a number of existing results.
We will say more on the Galoisian techniques involved in their proof afterward,
in  \S\ref{sec:intro-theoretical-results}.

\subsection{Differential transcendence in combinatorics}

A key application of our work is in the taxonomy of discrete structures.
A \emph{combinatorial class} is a set of discrete objects, each with a size, subject to the condition that the number of objects of a given size is finite. Given a class, enumerative data can be encoded in a formal power series in~$t$ such that the coefficient of $t^n$ is the number of objects of size $n$. This is the \emph{(ordinary) generating function (OGF) associated to the class}.   Classifying generating functions by the kind of equations they satisfy offers a meaningful organization of the structures themselves.

Stanley's call~\cite{stanley_differentiably_1980} to determine, and ideally characterize, those combinatorial classes whose generating functions satisfy linear differential equations with  polynomial coefficients (that is, the generating function is  D-finite)  spawned enormous activity in a variety of subdomains, notably pattern avoiding permutations (such as the Noonan-Zielberger  conjecture~\cite{noonan_enumeration_1996}) and lattice paths confined to the first quadrant~(the Bousquet--M\'elou-Mishna conjecture~\cite{bousquet-melou_walks_2010}). Group theory gives another example: The \emph{co-growth series} of a group is the generating function of excursions on its Cayley graph. It is directly related to the word problem associated to the group. The nature of the co-growth series (rational, algebraic, D-finite) reveals properties about the group, such as amenability~\cite{bell_complexity_2020}. A generating function analysis is a key pillar of a recent strategy~\cite{price_numerical_2019,elder_cogrowth_2012} to determine the amenability of Thompson's Group F, for example. It remains open to find a nice characterization of the set of combinatorial classes with a D-finite generating function, although such an understanding may happen in parallel with the resolution of a well-known, longstanding conjecture of Christol characterizing certain D-finite functions in terms of diagonals of multivariable rational functions~\cite{christol_globally_1990}.

There are a variety of techniques to decide the question of D-finiteness of a given series or function.  For instance, one can conclude that a function is not D-finite if the coefficient asymptotics of its Taylor expansion is not of a certain form~\cite{flajolet_non-holonomic_2004,flajolet_lindelof_2010,bousquet-melou_spanning_2015} or if the function has an infinite number of singularities~\cite{bousquet-melou_walks_2003, mishna_two_2009, beaton_consecutive_2017}. As we give a criterion for differential transcendence of series, we determine a new means to conclude that a function is not D-finite.

We use standard notation such as: $\C$ for the field of complex numbers, $\C(t)$ for the rational functions with complex coefficients and
$\C((t))$ for the field of (formal) Laurent series with coefficients in $\C$.
Throughout we identify $\C(t)$ with a sub-field of $\C((t))$, by identifying rational functions with their Taylor expansion.
We fix $R(t)\in t\C[[t]]$
a \emph{nonzero} power series, and we
define the following field endomorphism of $\C((t))$: 
\[
    \phir:\sum_nf_nt^n\mapsto
    f(R(t)):=\sum_nf_nR(t)^n\,.
    \]
This is well defined since $R$ has no constant term.
We make the following assumptions on~$R$, which we will comment upon later:
    \[\begin{array}{l}
            R(t)\in\C(t),~R(0)=0,~R'(0)\in\{0,1,\hbox{roots of unity}\},\\
            \hbox{but no iteration of $R(t)$ is equal to the identity.}
        \end{array}\leqno{(\cR)}
    \]
Our first main theorem is:

\begin{restatable}{mainthm}{mainzero}
\label{thmINTRO:main-0}
\change{Let $R(t)$ satisfy assumption $(\cR)$.
We suppose that there exist $a,b\in \C(t)$,
and $f\in\C((t))$ such that $f(R(t))=a(t)f(t)+b(t)$.
Then either $f$ is differentially transcendental over $\C(t)$
or there exists $\alpha,\beta\in\C(t)$ such that $f'=\alpha f+\beta$. }
\end{restatable}

For particular choices of $a$ and $b$ we obtain stronger results, namely:

\begin{restatable}{mainthm}{mainA}
\label{thmINTRO:a=1}
Let $R(t)$ satisfy assumption $(\cR)$.
We suppose that there exists $b\in \C(t)$ and $f\in\C((t))$ such that $f(R(t))=f(t)+b(t)$.
Then either $f\in \C(t)$, or~$f$ is differentially transcendental over $\C(t)$.
\end{restatable}

and:

\begin{restatable}{mainthm}{mainB}
\label{thmINTRO:b=0}
Let $R(t)$ satisfy assumption~$(\cR)$.
We suppose that there exist
$f\in\C((t))$ and $a\in\C(t)$, such that
$f(R(t))=a(t)f(t)$. Then either $f$ is algebraic over $\C(t)$ \change{and there exists a positive integer $N$ such that $f^N\in\C(t)$}, or $f$ is differentially transcendental over $\C(t)$.
\end{restatable}

\change{The weaker conclusion of the last theorem may be surprising, but we cannot hope for a stronger conclusion.
Indeed, if $R(t)=a+(t-a)S(t)^2$, with $a\in\C$ and $S(t)\in\C(t)$ chosen so that $(\mathcal{R})$ is satisfied, then
$y=(t-a)^{1/2}$ is a solution of $\phir(y)=S(t)y$, with coherent choice of square roots for $(t-a)$ and $S(t)^2$.
Moreover we notice that the statement of 
Theorem~\ref{thmINTRO:b=0} is compatible with the statement
of Theorem~\ref{thmINTRO:main-0}, 
indeed if $f^N\in\C(t)$ then $\frac{f'}{f}\in\C(t)$.}

\paragraph*{Complete trees.}
\label{sec:application-trees}
Rooted plane trees are fundamental objects in combinatorics and
computer science. We can prove the differential transcendency of the generating function for certain classes of trees
where the leaves are all at the same distance from the root; examples of such trees are in Fig.~\ref{fig:23trees}.
Rooted trees are directed from a single top vertex, or \emph{root}, down, and its children are ordered left to right. As we only consider finite trees, every path from the root ends at a vertex without descendants, which we call a \emph{leaf}.\footnote{The vocabulary is more intuitive if we view the tree upside down.} Here, the size of a tree is equal to the number of leaves that it has.

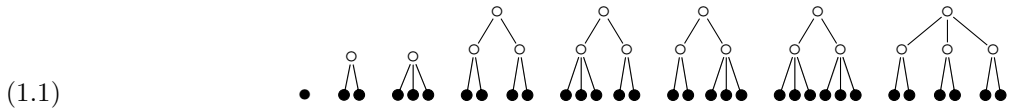
\begin{figure}[ht]
\begin{equation}
  \bullet\quad
\begin{tikzpicture} [style={fill,circle,inner sep=-.5pt, level distance=5mm,
  sibling distance=2mm}]
  \node {$\circ$}
         child {[fill] circle (2pt)}
         child {[fill] circle (2pt)};
\end{tikzpicture}\quad
    \begin{tikzpicture} [style={fill,circle,inner sep=-.5pt, level distance=5mm,
  sibling distance=2mm}]
  \node {$\circ$}
         child {[fill] circle (2pt)}
         child {[fill] circle (2pt)}
         child {[fill] circle (2pt)};
\end{tikzpicture}\quad
\begin{tikzpicture} [mytree]
  \node {$\circ$}
  child {node {$\circ$}
    child {[fill] circle (2pt)}
    child {[fill] circle (2pt)}{}}
  child {node {$\circ$}
    child {[fill] circle (2pt)}
    child {[fill] circle (2pt)}{}};
\end{tikzpicture}\quad
\begin{tikzpicture} [mytree]
  \node {$\circ$}
  child {node {$\circ$}
    child {[fill] circle (2pt)}
    child {[fill] circle (2pt)}
    child {[fill] circle (2pt)}{}}
  child {node {$\circ$}
    child {[fill] circle (2pt)}
    child {[fill] circle (2pt)}{}};
\end{tikzpicture}\quad
\begin{tikzpicture} [mytree]
  \node {$\circ$}
  child {node {$\circ$}
    child {[fill] circle (2pt)}
    child {[fill] circle (2pt)}{}}
  child {node {$\circ$}
     child {[fill] circle (2pt)}
    child {[fill] circle (2pt)}
    child {[fill] circle (2pt)}{}};
\end{tikzpicture}\quad
\begin{tikzpicture} [mytree]
  \node {$\circ$}
  child {node {$\circ$}
    child {[fill] circle (2pt)}
     child {[fill] circle (2pt)}
    child {[fill] circle (2pt)}{}}
  child {node {$\circ$}
     child {[fill] circle (2pt)}
    child {[fill] circle (2pt)}
    child {[fill] circle (2pt)}{}};
\end{tikzpicture}\quad
\begin{tikzpicture} [mytree]
  \node {$\circ$}
  child {node {$\circ$}
    child {[fill] circle (2pt)}
    child {[fill] circle (2pt)}{}}
   child {node {$\circ$}
    child {[fill] circle (2pt)}
    child {[fill] circle (2pt)}{}}
  child {node {$\circ$}
    child {[fill] circle (2pt)}
    child {[fill] circle (2pt)}{}};
\end{tikzpicture}
\end{equation}
\caption{All complete $\{2,3\}$-trees up to size 6}
\label{fig:23trees}
\end{figure}

Fix $\cS$, a set of positive integers each greater than 1.
The combinatorial class $\mathcal{T}$ of (unlabelled) $\cS$-trees is the set of rooted trees with the property that any vertex in the tree is either a leaf (with no descendants) or has $k$ children with $k$ from the set $\cS$. Were $\cS$ to contain 1, the class would contain an infinite number of objects of size 1 counter to the definition of combinatorial class.
\par
For example, every internal vertex in a  $\{2,3\}$-tree has either 2 or 3 children.  $\cS$-Trees defined in this way are extremely well studied from an enumerative perspective (see ~\cite[I.V.1]{flajolet_analytic_2009} and references therein) in part because the generating functions satisfy a polynomial equation
coming directly from a combinatorial recurrence:
Let $S(t)$ be the polynomial generating function for $\cS$:
$S(t)=\sum_{k\in\cS} t^k$ and $T(t)=\sum_{n=0}^\infty t_n t^n$, where
$t_n$ is the number of $\cS$-trees with $n$ leaves. Then,
\begin{equation}
T(t) = t+S(T(t)).
\end{equation}
For example, the generating function for the set of all $\{2,3\}$-trees enumerated by the number of leaves starts
  $T(t)=t+t^2+3t^3+O(t^4)$ and satisfies $T(t)=t+T(t)^2+T(t)^3$.

By adding a simple restriction, we find iterative equations.
A rooted $\cS$-tree is said to be \emph{complete} if it has the additional property that all leaves are at the same distance from the
root. Given an $\cS$-tree class $\mathcal{T}$, let $\mathcal{T}^c$ denote the sub-class of complete trees in $\mathcal{T}$,
where again we have omitted the dependence on $\mathcal{S}$ in the notation, for the sake of simplicity.
Flajolet and Sedgewick~\cite[Section I.6.2]{flajolet_analytic_2009} detail a substitution operator that permits a combinatorial description of the trees, and how to deduce a functional equation for the generating function of~$\mathcal{T}^c$.  The substitution operation is denoted here using square brackets and inside we indicate the substitution. Combinatorial substitution is reflected in the generating function with functional substitution.
The principle is best seen with an example. The class $\mathcal{T}^c$ of complete $\{2,3\}$-trees satisfies the recursive specification:
\begin{equation}\label{eq:23treecomb}
  \mathcal{T}^c\equiv \mathcal{\bullet} + \mathcal{T}^c\left[ \bullet\mapsto%
\begin{tikzpicture} [style={fill,circle,inner sep=-.5pt, level distance=5mm,
  sibling distance=2mm}]
  \node {$\circ$}
         child {[fill] circle (2pt)}
         child {[fill] circle (2pt)};
\end{tikzpicture}\, + \,
    \begin{tikzpicture} [style={fill,circle,inner sep=-.5pt, level distance=5mm,
  sibling distance=2mm}]
  \node {$\circ$}
         child {[fill] circle (2pt)}
         child {[fill] circle (2pt)}
         child {[fill] circle (2pt)};
\end{tikzpicture}\right].
\end{equation}
This combinatorial equation is read as every tree in $\mathcal{T}^c$ is isomorphic to either a single vertex, or a smaller tree in $\mathcal{T}^c$ where every leaf vertex is replaced with one of \begin{tikzpicture} [style={fill,circle,inner sep=-.5pt, level distance=5mm,
  sibling distance=2mm}]
  \node {$\circ$}
         child {[fill] circle (2pt)}
         child {[fill] circle (2pt)};
\end{tikzpicture}\, or \,
    \begin{tikzpicture} [style={fill,circle,inner sep=-.5pt, level distance=5mm,
  sibling distance=2mm}]
  \node {$\circ$}
         child {[fill] circle (2pt)}
         child {[fill] circle (2pt)}
         child {[fill] circle (2pt)};
\end{tikzpicture}.
Another way to view this is that the substitution generates a set of trees (of height one more) from a given tree. If that tree has~$n$ leaves, then the generating function for the set of trees (of height one more) that arise from all possible substitutions of its leaves is $(t^2+t^3)^n$. As there  $t^c_n$ trees with $n$ leaves, the generating function of all trees can be written
$\sum_{n\geq 0}t^c_n (t^2+t^3)^n$.
Remark,  the polynomial $t^2+t^3$ has no constant term, so the composition $T^c(t^2+t^3)$ makes sense if we interpret it as $T^c(t^2+t^3):=\sum_{n\geq 0}t^c_n (t^2+t^3)^n$. We conclude the generating function $T^c(t)$ for the counting sequence of complete $\{2,3\}$-trees satisfies the functional equation
\begin{equation}\label{eq:23treeeqn}
T^c(t)= t+ T^c(t^2+t^3),
\end{equation}
and has initial Taylor series expansion \[T^c(t)=t+t^2+t^3+t^4+2t^5+2t^6+O(t^7).\qquad(\text{OEIS} A014535)\]
Equation~\eqref{eq:23treeeqn} has an important consequence.
\begin{theorem}\label{thm:23tree-main}
The ordinary generating function for (unlabelled) complete rooted plane $\{2,3\}$-trees is differentially transcendental.
\end{theorem}
We prove this result using Theorem~\ref{thmINTRO:a=1}, or rather
its useful corollary, applied to the generating series associated with complete trees (see Corollary~\ref{cor:trees-chainsaw} below):

\begin{restatable}{cor}{maincortrees}\label{corINTRO:trees-chainsaw}
Let $R\in t^2\C[t]\setminus\{0\}$ and $b\in t\C[t]$, with $b\neq 0$ and $\deg_tb<\deg_tR$.
If there exists $f\in\C((t))$ such that $\phir(f)=f+b$, then $f$ is differentially transcendental over $\C(t)$.
\end{restatable}

\par
Remark, for general complete $\cS$-trees, the argument is the same:
the generating function $T^c(t)=\sum_{n\geq 0}t^c_n t^n$ of $\mathcal{T}^c$ satisfies
\[
T^c(t)=t+T^c(S(t)).
\]
\begin{theorem}\label{thm:tree-main}
Let $\cS$ be a finite set of positive integers each greater than 1. Then the ordinary generating function for (unlabelled) complete rooted plane $\cS$-trees is differentially transcendental.
\end{theorem}

 For example, for any~$m$, the set of complete $\{\lceil m/2\rceil, \dots, m\}$-trees is the well known class of \emph{B-trees of order $m$}~\cite{bayer_organization_1970}. By Theorem~\ref{thm:tree-main}, the generating function for B-trees of order $m$ is differentially transcendental.  Recall a $\cS$-tree class is only well defined if $\cS$ does not contain~1, so the restriction in Corollary~\ref{corINTRO:trees-chainsaw} is quite natural.

 A more general result is also true since the ordinary generating function of any class of complete $\cS$-trees is either differentially transcendental or rational by Theorem~\ref{thmINTRO:a=1}.
 Under some weak constraints on $\cS$, the asymptotic analyses of Odlyzko~\cite{Odlyzko_periodic_1982}  and de Bruijn~\cite{de_bruijn_asymptotic_1979} can be used to show the generating function is not rational.

 \begin{remark}
 Note the generating function $T(t)$ for all $\cS$-trees and the generating function $T^c(t)$ for complete $S$-trees satisfy equations that bear a superficial similarity:
$T(t) = t+ S(T(t))$ vs. $T^c(t) = t + T^c(S(t))$. However, the first is algebraic and the second differentially transcendental! Are there other examples of classically algebraic objects for which the addition of a simple condition changes the nature of the generating function in such a striking manner?
\end{remark}
\paragraph*{Random walks on self-similar graphs.}
Our second family of discrete objects concerns walks on self-similar graphs. We present the results for Sierpi\'nski graph, but similar conclusions are true for the entire class of \emph{symmetric self-similar graphs}, as described by B\"ohn and Teufl~\cite{kron_asymptotics_2004}.

The Sierpi\'nski graph results from a fractal generating process starting with a single line, iteratively rewritten and rescaled in particular way.
More precisely, one starts with a unit line, $S_0 = \begin{tikzpicture}
\draw [black, fill=white]
[l-system={Sierpinski triangle, step=12pt, angle=-120, axiom=X, order=0}]
lindenmayer system;
\end{tikzpicture}$ and applies the following replacement rule:
\begin{equation}
\begin{tikzpicture}
\draw [black, fill=white]
[l-system={Sierpinski triangle, step=12pt, angle=-120, axiom=X, order=0}]
lindenmayer system;
\end{tikzpicture}\quad\mapsto
\begin{tikzpicture}
\draw [black, fill=white]
[l-system={Sierpinski triangle, step=12pt, angle=-120, axiom=X, order=1}]
lindenmayer system;
\end{tikzpicture}
\end{equation} in an iterated process. Figure~\ref{fig:Sir-iteration} demonstrates the first few iterates. The \emph{Sierpi\'nski graph} is the limit of this process.
\begin{figure}[htbp]
    \centering
$S_0=$
\begin{tikzpicture}
\draw [black, fill=white]
[l-system={Sierpinski triangle, step=3pt, angle=-120, axiom=X, order=0}]
lindenmayer system;
\end{tikzpicture} \quad
 $S_1=$
\begin{tikzpicture}
\draw [black, fill=white]
[l-system={Sierpinski triangle, step=3pt, angle=-120, axiom=X, order=1}]
lindenmayer system;
\end{tikzpicture} \quad
 $S_2=$
\begin{tikzpicture}
\draw [black, fill=white]
[l-system={Sierpinski triangle, step=3pt, angle=-120, axiom=X, order=2}]
lindenmayer system;
\end{tikzpicture} \quad
   $S_3=$
\begin{tikzpicture}
\draw [black, fill=white]
[l-system={Sierpinski triangle, step=3pt, angle=-120, axiom=X, order=3}]
lindenmayer system;
\end{tikzpicture} \quad
 $S_4=$
\begin{tikzpicture}
\draw [black, fill=white]
[l-system={Sierpinski triangle, step=3pt, angle=-120, axiom=X, order=4}]
lindenmayer system;
\end{tikzpicture}
    \caption{Initial iterates defining the Sierpi\'nski graph. }
    \label{fig:Sir-iteration}
\end{figure}
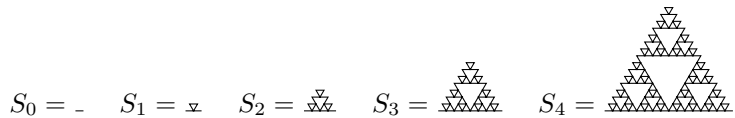

The \emph{Green function of a graph} is a probability generating function which describes the $n$-step displacement starting and returning to a certain origin vertex. The Green function of symmetric self-similar graphs satisfy homogeneous iterative equations,
\[
G(R(t))=a(t)G(t)
\]
with algebraic (often rational) $R$ and rational $a$.  Roughly, the substitution $t\mapsto R(t)$ has a combinatorial interpretation reflecting the self-similarity of the graph \cite{kron_asymptotics_2004}. As the graph is 4-regular, $G(4t)$ is the generating function for walks that begin and end at the origin on the Sierpi\'{n}ski
graph. These walks are also known as \emph{excursions} on the graph. The series begins:
\begin{equation*}
G(4t)=1+4\, t^{2}+4\, t^{3}+32\, t^{4}+76 t^{5}+348\, t^{6}+1112\, t^{7}+O(t^8).
\end{equation*}
Figure~\ref{fig:Sgraph} illustrates an example excursion.
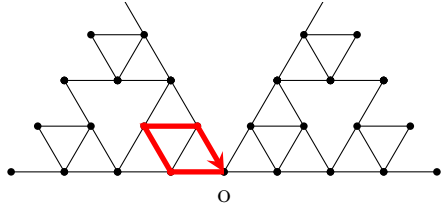
\begin{figure}[h]
    \centering
\begin{tikzpicture}[scale=.7]
\clip (-0.5,-0.6) rectangle (8.1,3.2);

    \draw  (0,0) [fill=white]
    [l-system={Sierpinski triangle, axiom=X, step=\trianglewidth/(2^3), order=3, angle=-120}]
    lindenmayer system;;
   \def\nx{7}
   \def\ny{3} \pgfmathsetmacro\nyy{(2+2*\ny)*sin(60)}
     \foreach \j in {0,...,\ny} {
            \foreach \i in {\j,...,\nx} {
                \path (0:\i) ++(60:\j) ++(120:\j) node[d] {} --++(60:2) node[d] {} --++(-1,0) node[d] {} --++(-60:1) node[d] {} --++(-60:1) node[d] {};
            }
        }
     \path (0:3) ++(60:2)node[g] {} ++(60:1) node[g] {} --++(-1,0)node[g] {};
     \path (0:1) ++(120:2)node[g] {} ++(60:2) node[g] {};
     \path (0:1) ++(120:2)node[g] {} ++(60:2) node[g] {};
      \path (0:7) ++(60:2)node[g] {} ++(120:1) node[g] {} ++(120:1) node[g] {} ++(1,0) node[g] {};
   \draw[red, line width=2pt, -stealth] (0:4) --++(-1,0) node[f] {} --++(120:1) node[f] {} --++(1,0) node[f] {} --++(-60:1) node[] {};
\node[label=below:o] at (0:4) {};
\end{tikzpicture}
    \caption{A close up on the origin (labelled o) of the Sierpinski Graph. The (red) path in bold is one of the 32 excursions of length 4.}
    \label{fig:Sgraph}
\end{figure}


Grabner and Woess~\cite[Proposition~1]{grabner_functional_1997} proved that the Green function~$G(t)$ for walks that return to their origin on the Sierpi\'{n}ski graph satisfies the functional equation
\begin{equation}\label{eq:green}
   G\left(\frac{t^2}{4-3t}\right)= \frac{(2+t)(4-3t)}{(4+t)(2-t)}\,G(t).
\end{equation}
We apply Theorem~\ref{thmINTRO:b=0} to $G(t)$.
\begin{theorem}\label{thm:green}
The Green function $G(t)$ of walks that start and end at the origin on the infinite Sierpi\'{n}ski graph is differentially transcendental over $\mathbb{C}(t)$.
\end{theorem}
\par
In light of Equation~\eqref{eq:green}, to prove Theorem~\ref{thm:green}, it suffices to show that $G(t)$ is not algebraic. Grabner and Woess in \emph{loc.cit.} show that the coefficient of $t^n$ in $G(t)$ grows asymptotically like \[n^{-\log 3/\log 5} F (\log n/\log 5)\] as $n$ goes to infinity, for some nonconstant periodic function $F$. The constant $-\log 3/\log 5$ is related to the fractal dimension of the underlying structure. Since the exponent of~$n$ is not rational,~$G(t)$ is not algebraic (see \cite[Theorem~VII.8]{flajolet_analytic_2009}), hence it is differentially transcendental.

One can apply Theorem~\ref{thmINTRO:b=0} to deduce the differential transcendence of the Green functions of excursions on other self similar graphs. The step of excluding the case of an algebraic generating function may follow from results of Teufl~\cite{teufl_average_2003} on the coefficient asymptotics, particularly when paired
with~\cite[Theorem~VII.8]{flajolet_analytic_2009}.

\paragraph*{Pattern avoiding permutations.}
The Noonan-Zeilberger conjecture~\cite{noonan_enumeration_1996} posited that the set of permutations avoiding  a fixed set of patterns should have a D-finite generating function, that is, that the generating function should satisfy a linear differential equation with rational coefficients.  Twenty years later, after much activity on the problem,  Garrabrant and Pak~\cite{garrabrant_permutation_2016} finally disproved it showing in a proof that the generating function of permutations avoiding a particular set of 30000 patterns was not D-finite. The study of consecutive pattern avoidance (defined below) has a slightly different flavour, but is also a good candidate for a systematic analysis, and indeed Elizalde and Noy~\cite{elizalde_clusters_2012} comprehensively classified the nature of the generating function for all small patterns to length 4. Permutations avoiding consecutive patterns can naturally describe bases of shuffle algebras with monomial relations~\cite{dotsenko_shuffle_2013}.
Most permutations avoid a consecutive pattern and consequently, it turns out to be better to encode enumerative data using the  \emph{exponential generating functions (EGF)},
i.e., we associate to our counting sequence~$(f_n)_{n=0}^\infty$
the series $\widehat{f}(t)=\sum_{n\geq 0} \frac{f_n}{n!}t^n$. To avoid confusion when were are talking about both EGF and OGF, the carat identifies the EGF of a sequence. A non-D-finite example was found quickly, and although it had been conjectured that in fact the reciprocal should always be D-finite, Elizalde and Noy gave strong evidence that the reciprocal of the EGF for the class of permutations avoiding the consecutive pattern $1432$ was not D-finite, an example we consider now.

A permutation $\sigma\in \mathfrak{S}_n$ is said to avoid the consecutive pattern 1423 if there is no $1\leq i\leq n-4$ such that $\sigma(i)<\sigma(i+4)<\sigma(i+2)<\sigma(i+3)$. Let $\widehat{P}(t)$ be the exponential generating function for permutations that avoid the consecutive pattern 1423. (OEISA201692). Elizalde and Noy~\cite{elizalde_clusters_2012} determined the following system of equations for the EGF $\widehat{P}(t)$:
\begin{equation}
    \widehat{P}(t)=\frac{1}{2-\widehat{S}(t)}\quad \text{ such that } \quad S(t)=S\left(\frac{t}{1+t^2}\right)\frac{t}{1+t}+1.
\end{equation}
 The non-D-finiteness of $S(t)$ was subsequently proved by Beaton, Conway and Guttmann in~\cite{beaton_consecutive_2017}
 who showed that an explicit solution to the functional equation had an infinite number of singularities. Theorem~\ref{thmINTRO:main-0} above,
 with $R(t)=\frac{t}{1+t^2}$ gives a potentially simpler path to establish that $S(t)$ is not D-finite
 (and indeed the even stronger conclusion that it is differentially transcendental) since you
 would just need to show that $S(t)$ is not solution of an inhomogeneous linear differential equation of order $1$. 
 As it is, as Beaton \emph{et al.} did establish
 that $S(t)$ is not D-finite, and hence we can conclude by Theorem~\ref{thmINTRO:main-0} that it is differentially transcendental.
 Regrettably, even with this stronger result we do not have enough information to use this to say something about $\widehat{P}(t)$ since we cannot conclude that $\widehat{S}(t)$ is differentially transcendental from the fact that $S(t)$ is differentially transcendental.\footnote{Indeed, there is a tempting conjecture of Pak and Yeliussizov~\cite{pak_complexity_2018} that would imply that for any integer sequence, if both the ordinary and exponential generating functions are differentially algebraic then actually, they are both  D-finite. Perhaps consecutive pattern avoiding permutations can be useful to support the result or find counter-examples.}
\par
Beaton, Conway and Guttmann~\cite{beaton_consecutive_2017} extended their analysis to an entire family. Let $\widehat{P_m}$ be the EGF for permutations that avoid the consecutive pattern $1m23\dots (m-2)(m-1)$ for any $m\geq 4$. Then,
\begin{equation}
    \widehat{P_m}(t)=\frac{1}{2-\widehat{S_m}(t)}\quad \text{ such that } \quad S_m(t)=S_m\left(\frac{t}{1+t^{m-2}}\right)\frac{t}{1+t}+1.
\end{equation}
They deduce that for any $m$, as $S_m(t)$ is not D-finite, the reciprocal of the generating function  $\widehat{P_m}(t)$ is also not D-finite. Applying Theorem~\ref{thmINTRO:main-0}, we conclude that for every $m$,  $S_m(t)$ is furthermore differentially transcendental.

\par
As a partial result towards  the proof of Theorem~\ref{thmINTRO:main-0}, we obtain (see Corollary~\ref{cor:key-Galois-theory-bis} below):

\begin{restatable}{mainthm}{mainG}
\label{thmINTRO:main-R-algebraic}
We assume that $R$ satisfies $(\cR)$.
Let us consider an equation of the form $\phir(y)=ay+b$, with $R,a,b\in\C(t)$ and $a,b\neq 0$.
Let $f\in  \C((t))$
satisfy the equation $\phir(f)=af+b$.
If the functional equation $\phir(y)=ay$ does not have a
nonzero algebraic solution, then either $f\in\C(t)$ or
$f$ is differentially transcendental over $\C(t)$.
\end{restatable}
It follows that  one could consider to prove the $D$-finiteness of $S_m(t)$, with $m\geq 4$, proving that it is not a rational function and that the functional equation
    \[
    y\left(\frac{t}{1+t^{m-2}}\right)=\frac{1+t}{t} y(t)
    \]
does not have a nonzero algebraic solution.

\subsection{A summary of the main theoretical results}
\label{sec:intro-theoretical-results}
The interactions between Galois theory of functional equations and enumerative combinatorics has produced spectacular results in the last few years and we are most likely far from having explored all the possible applications and interaction. One potential barrier to a more steady flow of collaborations is the relatively opposing perspectives towards mathematics between the two communities.
Indeed, while generally in enumerative combinatorics algebraic abstraction comes into the picture only when strictly necessary, in the Galois theory of functional equations one frequently finds theorems that are stated over very general fields. These results may actually apply to very concrete situations, yet the novice user can find the theorem formulations daunting.
For this reason, we have made the choice to try to write the functional theoretical part of the paper in a ``bottom-up'' spirit.  We have tried to introduce an abstract point of view as late as possible, to clarify where and why a more sophisticated setting is necessary in order to be accessible to a larger audience of readers.
\par
Our main purpose is to prove Theorem~\ref{thmINTRO:main-0} (as well as Theorem~\ref{thmINTRO:a=1} and Theorem~\ref{thmINTRO:b=0}),
that we recall here:

\mainzero*

These statements are related to a collection of results in the literature that establish a dichotomy ``algebraic vs differentially trascendental'' for the following kind of functional equations:
\begin{enumerate}
    \item If $R(t)=\frac{t}{t+1}$ then the result is proved in \cite[Theorem~2]{Nishioka_note_1984}.
    Notice that in \emph{loc.cit.} the author writes the operator using the variable $s=\frac{1}{t}$, so that
    $\phir(f(s))=f(s+1)$.
    The case $R(t)=\frac{t}{t+1}$ is considered also in~\cite{bostan_differential_2020},
    where the authors establish the differential transcendence of solutions
    over the germs of meromorphic functions at zero, inspired by the example of the generating
    series of Bell numbers in~\cite{klazar_bell_2003}, and make a list of generating functions coming from combinatorics
    which satisfy functional equations associated with such an $R$.
    \item
    In the case $R(t)=t^d$, with $d\geq 2$, i.e., in the so-called Mahler case, the result above is proved in
    \cite[Theorem~3]{Nishioka_note_1984}, \cite[page~22]{rande_mahler_1992} and in \cite{Nguyen_Hypertranscedance_2011}. A well-known application is the generating function of the celebrated Thu\"e-Morse sequence (OEISA010060) which satisfies an order one Mahler equation with $d=2$. It is not rational, hence it is differentially transcendental over $\mathbb{C}(t)$.
\end{enumerate}
The case $R(t)=qt$, with $q\in\C$, $q\neq 0,1$, and not a root of unity,
    is proved in \cite[Theorem~1.2]{Ishizaki_Hypertranscendency_1998} and \cite[Theorem~2]{Ogawara_Differential_2014}.
    The classification of walks in the quarter plane is the main example from combinatorics to which these results have been applied.
    See \cite{dreyfus_nature_2018}.
Notice that we do not cover the case of a general nonzero~$R'(0)$, different from a root of unity, which is most likely true, but more complicated to unravel. We haven't found any functional equation coming from combinatorics  associated to a rational function~$R$ of this form and for this reason we have left this case aside for the time being. We will comment again on this precise issue.
\par
The papers \cite{ramis_about_1992}, \cite{bezivin_solutions_1993} and \cite{schafke_consistent_2019}
are a first step toward the generalization of the results above: in them the authors prove the
rationality or the algebraicity of solutions of systems of linear functional equations of different natures.
They consider situations that are quite similar to the conclusion of Theorem~\ref{thmINTRO:main-0}, namely the fact that
if $f$ is differentially algebraic than it is solution of the system
    \[\phir(f)=af+b, ~f'=\alpha f+\beta.\]
In \emph{loc. cit.}, the authors prove that, when $R(t)\in\{t+1, qt, t^d\}$, a solution of the system above is necessarily algebraic, and even rational in some cases.
Therefore, it is quite natural to ask whether a similar result holds for a general $R(t)$.

\par
The first Galoisian approach to these kind of statements is developed in \cite{hardouin_hypertranscendance_2008},
which has inspired different parameterized Galois theories leading to similar statements in different settings.
See \cite[\S3.1]{hardouin_differential_2008}, \cite[\S3]{di_vizio_difference_2017}, \cite[\S3]{ovchinnikov_galois_2015}.
The main result of \cite{adamczewski_hypertranscendence_2021} generalizes
the results obtained for functional equations of the form $\phir(y)=ay+b$, when $R(t)=t+1,\frac{t}{1+t},qt,t^d$ to higher order functional equations.

\paragraph*{Proof of Theorem~\ref{thmINTRO:a=1} and Theorem~\ref{thmINTRO:b=0}}
In \S\ref{sec:mainresults-1}, we start proving the following statement:

\mainA*

The aforementioned Corollary~\ref{corINTRO:trees-chainsaw} is derived from
Theorem~\ref{thmINTRO:a=1}. A second result is the following:

\mainB*

They are proved using a  Galoisian result as a black box (see Proposition~\ref{prop:DiffAlgHanoi} below).
They do not require many precautions in their proof because in both cases
$f$ determines the whole space of solutions of the functional equation,
up to an additive or a multiplicative constant.
The rationality of the coefficients involved allows a relatively elementary reasoning on
the residues of the expressions we obtain.
\par
A major ingredient of the proof is the main theorem of \cite{becker_hypertranscendency_1995}
(see Theorem~\ref{thm:Becker-Bergweiler} below),
which classifies
rational $R$ satisfying $(\cR)$ for which there exists a rational $m(t)$ such that
$R(m(t))=m(t^d)$, for some $d\geq 2$.
Those cases are \emph{a priori} eliminated from our consideration, as they are reduced to Mahler equations which have previously been considered.
For all the others, the solutions of the iterative logarithmic equation
$y(R(t))=R'(t) y(t)$ are differentially transcendental.
We call $\F$ the field generated by $\bF$, a non zero solution $\Psi$ of $y(R(t))=R'(t) y(t)$, whose existence is discussed in \S\ref{sec:iterative-log}, and all its derivatives.
Therefore, under the extra assumption:
\begin{quote}\mbox{}
\begin{minipage}{.1\textwidth}$(\cDT)$\end{minipage}
\begin{minipage}{.7\textwidth}
We fix $R\in\C(t)$, with $R(0)=0$, such that no iteration of $R$ is the identity.
We suppose that there exists $\Psi\in\F$ such that
$\phir(\Psi)=R'(t)\Psi$ and that $\Psi$ is differentially transcendental over~$\C(t)$ with respect to $\frac{d}{dt}$.
\end{minipage}
\end{quote}
we reduce the proof to the following statements:

\begin{restatable}{mainthm}{mainD}
\label{thmINTRO:main-additive}
Under the assumptions $(\cDT)$, let $b\in\C(t)$ and let $f\in\C((t))$ be a solution of the equation $\phir(f)=f+b$.
Then, either $f\in\C(t)$ or $f$ is differentially transcendental over $\C(t)$, with respect to $\frac{d}{dt}$.
\end{restatable}

\begin{restatable}{mainthm}{mainE}
\label{thmINTRO:main-multiplicative}
Under the assumptions $(\cDT)$, let $a\in\C(t)$ and let $z\in\C((t))$
be such that $\phir(z)=az$.
\change{Then either $z$ is algebraic over $\C(t)$ and
there exists a positive integer $N$ such that $z^N\in\C(t)$ or is differentially transcendental over $\C(t)$.}
\end{restatable}

We prove
Theorems~\ref{thmINTRO:main-additive} and \ref{thmINTRO:main-multiplicative},
in Section \S\ref{sec:main-additive} and \S\ref{sec:main-homogenous}, respectively.
As far as assumptions $(\cDT)$ are concerned, the assumption on $R'(0)$ has disappeared and has been
replace by an assumption on $\Psi$. The point is that when $R'(0)\neq 0$ is not a root of unity we
have more than two cases: either the functional equation reduces to a $q$-difference equation, which is already covered by the literature,
or $\Psi$ is differentially transcendental, or $\Psi$ is differentially algebraic of certain types classified by Ritt in \cite{ritt_transcendental_1926}.
The latter situation is not treated in this paper.
See Remark~\ref{rmk:ritt-missed-case}.
\par
We show in Remark~\ref{rmk:Beker-Bergweiler} that the polynomial
$t^2+t^3$ associated with the complete $\{2,3\}$-trees (see Equation~\eqref{eq:23treeeqn})
is an example for which $\Psi$
is differentially transcendental and cannot be reduced to a Mahler equation,
hence our results are really needed in the applications.
See Proposition~\ref{prop:novelty}.

\paragraph{Proof of Theorem~\ref{thmINTRO:main-0}.}
Theorems~\ref{thmINTRO:b=0} and \ref{thmINTRO:a=1} are not enough to prove Theorem~\ref{thmINTRO:main-0}.
Now, if we can find an algebraic solution $z$ of $y(R(t))=a(t)y(t)$,
we can, by dividing the unknown function by $z$, transform the equation $y(R(t))=a(t)y(t)+b(t)$ into an equation that resembles that of Theorem~\ref{thmINTRO:a=1} except now with \emph{algebraic} coefficients:
$y(R(t))=y(t)+\frac{b(t)}{z(R(t))}$.
We  are no longer in the situation of Theorem~\ref{thmINTRO:a=1}.
Furthermore, Theorem~\ref{thmINTRO:main-0} provides one
formal power series solution of $y(R(t))=a(t)y(t)+b(t)$, but no solution of the associated homogeneous functional equation $y(R(t))=a(t)y(t)$, that we need to construct abstractly, to obtain a full set of solutions. 
To fill this gap,  we consider a more abstract setting and
prove statements that can be considered as particular cases of Theorem~\ref{thmINTRO:main-0}:
In other words, we cannot prove Theorem~\ref{thmINTRO:main-0} directly, but we need to make a Galoisian detour.
\par
Let $\bF$ be the field of Puiseux  series
with complex coefficients, i.e., the algebraic closure of $\C((t))$, and let
$\bK$ be the algebraic closure of $\C(t)$ in $\bF$.
To extend the endomorphism of composition with $R(t)$ to $\bF$ one has to
choose a compatible system of roots of $R$, and hence implicitly allow $R$ to be an element of $\bK$.
We call $\Falg$ the field obtained adjoining a convenient
solution $\Psi$ of $y(R(t))=R'(t)y(t)$ and all its derivatives to $\bF$.
Under the following assumption:
\begin{quote}\mbox{}
    \begin{minipage}{.1\textwidth}$(\cDT_{alg})$\end{minipage}
    \begin{minipage}{.7\textwidth}
        We fix a nonzero $R\in\bK\cap t\C[[t]]$, such that no iteration of $R$ is equal to the identity.
        We suppose that $\Psi\in \Falg$ is differentially transcendental over $\bK$
        with respect to $\frac{d}{dt}$.
    \end{minipage}
\end{quote}
we prove two statements analogous to Theorem~\ref{thmINTRO:a=1} and Theorem~\ref{thmINTRO:b=0} (see Proposition~\ref{prop:main-additive-ALG}
and Proposition~\ref{prop:maint-multiplicative-ALG}, below).
Finally we put all the pieces together and we  get to the only
point where we actively need some Galois theory.
Using the abstract solutions of $y(R)=ay$ and $f$, we finally have a
$\C$-algebra where $f(R)=af+b$ has a full set of solutions.
We need to find a nontrivial automorphism of it, commuting with the derivation and the
composition with $R$: $\C$ is a too small field to guarantee such a property and
in \S\ref{subsec:GaloisTheory} we have to consider a even more general situation
by performing an extension of constants to a differentially closed field.
This leads to the proof of Theorem~\ref{thmINTRO:main-R-algebraic}.
Ultimately, this more general setting allows us to prove Theorem~\ref{thmINTRO:main-0},
in its very elementary formulations, that, as we hopefully have convinced the reader,
apply neatly to different combinatorial situations.

\paragraph{Acknowledgments}
We are grateful to the Fondation Mathématique Jacques Hadamard that not only has sponsored
the post-doctoral position of the second author at the Laboratoire de Mathématiques de Versailles,
but was also ready to support Marni Mishna's visit to Versailles in the spring 2020.
The pandemic and our personal lives have delayed the project, until thanks to the joint CNRS-PIMS International Research Laboratory supported a visit of Lucia Di Vizio
at Simon Fraser University in Burnaby, Canada.
\par
We also thank Jason Bell, Xavier Buff, Vincent Guedj, 
Sarah Selkirk and Michael Singer for useful discussions and comments.

\section{Technical background}
\label{sec:technical}

We first briefly recall the few properties of differential fields that we need in this paper.
For the readers interested in knowing more, one of the classical references on the subject is \cite{ritt_differential_1966}. Then we examine the possible solutions to the functional equation $y(R(t))=R'(t)y(t)$ in order to precisely define the differential algebra that we use in the proof.
\subsection{A brief digression on differential algebra}
\label{subsec:DiffAlgebra}
Let $(K,\partial)$ be a \emph{differential field} of characteristic zero, that is a field $K$ of characteristic zero equipped with a linear map $\partial$ such that
$\partial(ab)=\partial(a)\,b+a\,\partial(b)$, for any $a,b\in K$. We call $\partial$ the derivative of $K$.
Most of the time we will write only $K$ for $(K,\partial)$.
A \emph{differential field extension} $F/K$ is a field extension such that both $K$ and $F$ are differential fields and the derivative of $F$ extends the derivative of $K$. We denote by $\partial$ both the derivative of $F$, and the one of $K$.

\begin{exa}
For instance, we can take $K=\C(t)$, $F=\C((t))$ and $\partial=\frac{d}{dt}$.
\end{exa}

An element $f\in F$ is said \emph{differentially algebraic over $K$}
if there exists an integer $n\geq 0$ and a polynomial $P$ in $n+1$ variables and with coefficients in $K$
such that $P(f,\partial(f),\dots,\partial^n(f))=0$.
We say that $f$ is \emph{differentially transcendental over $K$} if it is not differentially algebraic over $K$.
Sometimes, to put the accent on the choice of the derivative,
we will say that $f$ is differentially algebraic (resp. transcendental) over $K$ with respect to $\partial$.
We call $P(y,\partial(y),\dots,\partial^n(y))=0$ an algebraic differential equation (in the unknown function $y$, with coefficients in $K$, of order $n$ with respect to the derivative $\partial$, satisfied by $f$).

The field $F$ is said to be an \emph{algebraic differential extension of $K$} if any element of $F$ is differentially algebraic over $K$ and it is said to be a differentially transcendental extension of $K$ otherwise. A \emph{purely differentially transcendental  extension $F/K$} is a differentially transcendental extension
such that no element of $F\setminus K$ is differentially algebraic over $K$.
Notice that if $f\in F$ is differentially transcendental, then for any $n\geq 0$ the algebra $K[f,\partial(f),\dots,\partial^n(f)]$ is a purely transcendental algebra over $K$
and can naturally be identified with a ring of polynomials in $n+1$ variables.
\par
The following elementary lemma is the only property of differential extensions that we will need in what follows:
\begin{lemma}\label{lemma:DiffAlg}
Let $L$ be an intermediate differential field of the differential extension $F/K$,
i.e. an intermediate field stable by $\partial$. We suppose that $L/K$ is a differentially algebraic extension.
Then:
\begin{enumerate}
    \item If $\widetilde{K}$ is an intermediate differential field of $F/K$ that is a purely differentially transcendental extension of $K$, then $L\cap\widetilde K=K$.
    \item If $f\in F$ is differentially transcendental over $K$, then $f$ is also differentially transcendental over $L$.
\end{enumerate}
\end{lemma}
\begin{proof}
The first assertion follows from the definitions.
We prove now the second statement, under the assumption that $L$ is generated over $K$ by a finite number of
differentially algebraic elements and their derivatives, which implies that
$\hbox{tr.deg}_K L<\infty$.
We suppose by contradiction that $f$ is differentially algebraic over $L$,
i.e., that there exist an integer $n\geq 0$  and a polynomial $P$ with coefficients in $L$ and in $n+1$ variables,
such that $P(f,\partial(f),\dots,\partial^n(f))=0$.
Therefore:
    \[
    \hbox{tr.deg}_K K(\partial^i(f), i\geq 0)\leq
    \hbox{tr.deg}_K L(\partial^i(f), i\geq 0)=
    \hbox{tr.deg}_K L+\hbox{tr.deg}_L L(\partial^i(f), i\geq 0)<\infty.
    \]
We conclude that $f$ is differentially algebraic over $K$, against the assumptions.
\par
If $f$ is differentially algebraic over a general differentially algebraic extension $L$ of $K$,
then $f$ is differentially algebraic over the extension of $K$ generated by the coefficients of its differential equation
$P(f,\partial(f),\dots,\partial^n(f))=0$ and their derivatives.
We conclude from the previous discussion that $f$ is differentially algebraic over $K$, which completes the proof.
\end{proof}

\subsection{The iterative logarithm of \texorpdfstring{$R$}{R}}
\label{sec:iterative-log}

We recall the notation.
Let $\C((t))$ be the field of (formal) Laurent series with coefficients in the field $\C$ of complex numbers
and let  $R(t)\in t\C[[t]]$
be \emph{nonzero} power series, without constant coefficient.
We consider the field endomorphism $\phir:\C((t))\to\C((t))$ defined by
    \[
    f(t):=\sum_nf_nt^n\mapsto\phir(f(t)):=f(R(t)):=\sum_nf_nR(t)^n\,,
    \]
which is well defined since $R$ has no constant term.
We identify $\C(t)$ with a sub-field of $\C((t))$, by identifying rational functions with their Taylor expansion.
\par
For any $f\in \C((t))$, we will denote by $\frac{df}{dt}$ or simply by $f'$
 the usual derivative of $f$ with respect to~$t$. We will also write
 $f^{(n)}$ for $\frac{d^nf}{dt^n}$.
By chain rule,  we have $\phir(f)'=\phir(f')R'$.
We are interested in the (differential properties) of solutions of the functional equation
    \begin{equation}
    \label{eq:Julia}
    y(R(t))=R'(t)y(t),
    \end{equation}
usually named after Julia or Jabotinski in the literature,
which will play an auxiliary, yet crucial, role in the proofs below.
Its solutions are sometimes called iterative logarithms. See \cite[Page 31, Def. 4(b)]{ecalle_theorie_1974} or \cite{ecalle_theorie_1975}.
We recall assumption $(\cR)$ on $R$:
    \[
        \begin{array}{l}
            R(t)\in\C(t),~R(0)=0,~R'(0)\in\{0,1,\hbox{roots of unity}\},\\
            \hbox{but no iteration of $R(t)$ is equal to the identity.}
        \end{array}  
        \leqno{(\cR)}
    \]
First of all, we recall the main result of \cite{becker_hypertranscendency_1995}, which extends
Ritt's theorem from~\cite{ritt_transcendental_1926}
(see also \cite{fernandes_survey_2021}).
It mentions the $d$-th Chebyshev polynomial $T_d(t)$, for $d\geq 1$, defined by the
identity $T_d(\cos\theta)=\cos(d\theta)$, for any real value of $\theta$,
and homography, which in this context is essentially a fractional linear transformation.

\begin{thm}[{\cite{becker_hypertranscendency_1995}}]\label{thm:Becker-Bergweiler}
Under the assumption $(\cR)$ we have:
\begin{enumerate}
\item
If $R(t)$ has a zero of order $d\geq 2$ at $0$, there exists a solution $\tau\in t+t^2\C[[t]]$ of
the functional equation $\tau(R(t))=\tau(t)^d$ and
one of the following possibilities occurs:
\begin{enumerate}
    \item There exists a homography $m$ such that
    $R(m(t))=m(t^d)$ and $\tau$ is differentially algebraic over $\C(t)$;
    \item There exists a homography $m$ such that
    $R(m(t))=m(\pm T_d(t))$ and $\tau$ is differentially algebraic over $\C(t)$;
    \item $\tau$ is differentially transcendental over $\C(t)$.
\end{enumerate}
\item
If  $R'(0)=1$, there exists
 $\tau\in\C[[t]]$ satisfying the functional equation $\tau(R(t))=\tau(t)+1$ and $\tau$ is differentially transcendental over $\C(t)$.
\end{enumerate}
\end{thm}

\begin{proof}
\begin{enumerate}
\item  Let $d$ be the order of $R$ at $0$.
It follows from Böttcher's Theorem \cite{boettcher_principal_1904}, \cite[\S II]{ritt_iteration_1920} that
there exists a power series $\tau(t)\in t\C[[t]]$,
such that $\tau(R(t))=\tau(t)^d$.
Hence $\tau$ admits a compositional inverse~$\sigma$ which is also an element of
$t\C[[t]]$ and which is differentially transcendental over $\C(t)$
if and only if $\tau$ is differentially transcendental over $\C(t)$.
Then the result follows from \cite[Theorem~(ii), page 466]{becker_hypertranscendency_1995}.
\item 
This assertion coincides with \cite[Theorem~(iii), page 466]{becker_hypertranscendency_1995}.
\end{enumerate}
\end{proof}

\begin{rmk}\label{rmk:Beker-Bergweiler}
\begin{enumerate}
    \item
If $R$ has a zero of order $d\geq 2$ at $0$, then $\tau$ is called a Böttcher function.
Böttcher's theorem says more about $\tau$ than we have used in the proof above.
It guarantees that $\tau$ is convergent at $0$ and
is defined, in a neighbourhood of the origin, as the uniform limit
of the family $\left[R^{\circ n}(t)\right]^{1/d^n}$, when $n$ goes to infinity.
See \cite[\S II]{ritt_iteration_1920} for the details.
\par
If $R(t)=t^d$, then $\tau=\zeta t$, with $\zeta^{d-1}=1$.
If on the other hand, $R(t)=\pm T_d(t)$,
then $\tau$ is an algebraic function defined by
$\zeta^2\tau(t)^2-2\zeta t\tau(t)+1=0$, with $\zeta^{d-1}=\pm 1$.
In the latter case, the compositional inverse
$\sigma$ of $\tau$ verifies $\sigma(s)=\frac{1}{2}\left(\zeta s+\frac{1}{\zeta s}\right)$,
as one can verify applying the variable change $s=\tau(t)$.\footnote{
Notice that the polynomial $T_d(t)$ has a zero of order $d$ at $\infty$, while we have supposed that we have a fixed point at $0$.
Of course, it suffices to make a change of variable of the form $t\mapsto 1/t$.
\par
One can prove quite easily, by recurrence on $d$ that $T_d(\sigma(s))
=\sigma(s^d)$. Indeed we have $T_2(t)=2t^2-1$, so that for any root of unity $\zeta$
we have $T_2\left(\frac{1}{2}\left(\zeta s+\frac{1}{\zeta s}\right)\right)=\frac{1}{2}\left(\zeta^2 s^2+\frac{1}{\zeta^2 s^2}\right)$.
If $\zeta=1$, we have the conjugation for $d=2$. One completes the proof using the recurrence
$T_{d+1}(t)=2tT_d(t)-T_{d-1}(t)$.}
    \item
If $R'(0)=1$, the series $\tau$ is said to be an \emph{Abel function}.  In this case, $\tau$ is analytic in some ``petals'' centered at $0$, and admits an asymptotic expansion at zero.
\end{enumerate}
\end{rmk}

Now we switch our attention back to the Julia equation~\eqref{eq:Julia}.
Before actually stating the existence of its solutions we need to define a convenient function field where one can find them. In what follows we denote by $\C(\{t\})$ the field of germs of meromorphic functions at $0$.
We are only sketching the proof of the following lemma as it is quite classical:

\begin{lemma}\label{lemma:construction-F}
The tensor product $\C(\{t\})(\log t)\otimes_{\C(\{t\})}\C((t))$
is a domain,
equipped with a canonical action
of $\frac{d}{dt}$ and $\phir$, such that
$\frac{d}{dt}\circ\phir=R'(t)\phir\circ\frac{d}{dt}$.
\par
Its field of fractions, that we will denote by $\C((t))(\log t)$,
is equipped with an action of both $\frac{d}{dt}$ and $\phir$, such that $\frac{d}{dt}\circ\phir=R'(t)\phir\circ\frac{d}{dt}$.
Its sub-field of invariant elements with respect to $\phir$ is $\C$.
\end{lemma}

\begin{proof}[Sketch of the proof.]
First of all, notice that the differential equation $y'=\frac{1}{t}$ does not have any solution in the field of Puiseux series with coefficients in $\C$, hence any solution of such differential equation is transcendental over $\C((t))$.
Therefore we consider the transcendental extension $\C((t))(T)$
of $\C((t))$ and set $T'=\frac{1}{t}$. One can verify by hand
that its sub-field of constants is $\C$. This field can already be identified with $\C((t))(\log t)$,
but we need to define an action of $\phir$ on it, with the expected properties. That's the object of the lemma.
\par
We look at $\C(\{t\})$ as to the field of germs of meromorphic function at $0$ over the Riemann surface of the logarithm, so that
it makes sense to consider the field $\C(\{t\})(\log t)$.
Moreover, we identify $\C(\{t\})$ with a sub-field of $\C((t))$,
identifying the element of $\C(\{t\})$ with its Taylor expansion at zero.
Both $\C(\{t\})(\log t)$ and $\C((t))$ can be embedded in $\C((t))(T)$ and we can consider
the tensor product $\C(\{t\})(\log t)\otimes_{\C(\{t\})}\C((t))$ of $\C(\{t\})$-algebras,
equipped with the canonical derivative induced by $\frac{d}{dt}$,
i.e., $\frac{d}{dt}(f\otimes g)=f'\otimes g+f\otimes g'$,
for any $f\otimes g\in \C(\{t\})(\log t)\otimes_{\C(\{t\})}\C((t))$.
Therefore there is a natural morphism from
$\C(\{t\})(\log t)\otimes_{\C(\{t\})}\C((t))$ to
$\C((t))(T)$ defined by
    \[
    \sum_if_i(\log t)^i\otimes g_i
    \mapsto \sum_i f_ig_iT^i,
    \]
which is injective and commutes with the action of $\frac{d}{dt}$.
The endomorphism $\phir$ acts naturally on $\C(\{t\})(\log t)$, since in a small punctured neighborhood of $0$
on the Riemann surface of the logarithm
the function $\log(R(t))$ has an analytic meaning. Therefore we can
consider its canonical action on the tensor product, which is defined by $\phir(f\otimes g)=\phir(f)\otimes\phir(g)$,
for any $f\otimes g\in \C(\{t\})(\log t)\otimes_{\C(\{t\})}\C((t))$.
Finally, the chain rule holds on $\C(\{t\})(\log t)\otimes_{\C(\{t\})}\C((t))$, thanks to the Leibniz rule,
since it holds on $\C(\{t\})(\log t)$ and on $\C((t))$.
\par
We take $\C((t))(\log t)$ to be the field of fractions of $\C(\{t\})(\log t)\otimes_{\C(\{t\})}\C((t))$.
The sub-field of invariant elements of $\C((t))(\log t)$ with respect to
$\phir$ is still $\C$: One can prove it ``by hand'', however it is a known result under quite general assumptions.
See \cite[Lemma~2.13]{ovchinnikov_galois_2015},
which extends \cite[Lemma~1.8]{van_der_put_galois_1997}
to the case of difference equations with respect to an endomorphism.
\end{proof}

\begin{rmk}
\change{Some readers may like a more algebraic approach to the proof of Lemma~\ref{lemma:construction-F}. Indeed, there exists a series $\chi\in\C[[t]]$ such that 
$\tau(t)=t(1+t\chi(t))$, so that one can formally define 
$\log\tau=\log t+\log\left(1+t\chi(t)\right)$, 
using the definition of $\log(t)$ as a formal power series. One can show formally that the power series defining 
$\log(t)$ has the expected properties when applied to 
products, so that if $\phir(\tau)=\tau^d$ then $\phir(\log\tau(t))=d\log(\tau(t))$.}
\end{rmk}

We are ready to describe the solutions of Equation~\eqref{eq:Julia}.

\begin{prop}\label{prop:Psi}
Let $R(t)\in \C(t)$ satisfy $(\cR)$ and let $\Psi$ be a solution of $\phir(y)=R'(t)y$.
We have:
\begin{enumerate}
\item
If $R(t)$ has a zero of order $d\geq 2$ at $0$, then $\Psi\in\C((t))(\log t)$
and one of the following possibilities occurs:
\begin{enumerate}
    \item there exists a homography $m$ such that
    $R(m(t))=m(t^d)$ and $\Psi$ is a homography composed with a monomial times a logarithm.
    \item there exists a homography $m$ such that
    $R(m(t))=m(\pm T_d(t))$, where $T_d(t)$ is a Chebyshev polynomial, and
    $\Psi$ is in $\C(t,\log t)\subset\C((t))(\log t)$.
    \item $\Psi$ is differentially transcendental over $\C(t)$.
\end{enumerate}
\item
If  $R'(0)$ is a root of unity, then $\Psi\in\C((t))$ is
differentially transcendental over $\C(t)$.
\end{enumerate}
\end{prop}

\begin{proof}
If $\tau$ is a Böttcher function then $\frac{\tau'(R(t))}{\tau(R(t))}R'(t)=d\frac{\tau'(t)}{\tau(t)}$.
Therefore we need to find a solution of the functional equation $\phir(y)=dy$.
As we have pointed out in Remark~\ref{rmk:Beker-Bergweiler}, we know that $\tau$ is a convergent series at $0$, with a zero at $0$.
Therefore the logarithmic derivative $\frac{\tau'(t)}{\tau(t)}$
admits as primitive $\log\tau(t)\in\C(\{t\})(\log t)$,
which is analytic on the Riemann surface of the logarithm, in a
punctured neighborhood of $0$. We have:
$\phir(\log\tau(t))=\log\tau(R(t))=\log\tau(t)^d=d\log\tau(t)$.
We conclude that $\Psi=\frac{\tau}{\tau'}\log\tau\in \C((t))(\log t)$ is a solution of $\phir(y)=R'(t)y$.
The result follows from Theorem~\ref{thm:Becker-Bergweiler}.
\par
Let us prove the second statement.
If $R'(0)$ is a nontrivial root of unity, then it is enough to replace $R$ by an iteration of $R$:
For any $n\geq 1$ we set $R^{\circ 1}=R$
and $R^{\circ n}=R\circ R^{\circ(n-1)}$, so that,
using the chain rule, one can prove recursively that a solution of Equation \eqref{eq:Julia}
is also a solution of $y(R^{\circ n}(t))=(R^{\circ n}(t))'y(t)$.
Therefore, if $R'(0)$ is a root of unity, replacing $R$ with a convenient iterate of $R$, we can assume that $R'(0)=1$, and deduce the result from the case $R'(0)=1$.
If $\tau$ is an Abel function, then $\tau(R(t))=\tau(t)+1$ and hence $\phir(\tau')R(t)'=\tau'$.
$\Psi=\frac{1}{\tau'}$ is a formal power series solution of the functional equation $y(R(t))=R'(t)y$.
The result follows from Theorem~\ref{thm:Becker-Bergweiler}.
\end{proof}

\begin{exa}\label{exa:Psi}
The functional equation \eqref{eq:23treeeqn} of complete $\{2,3\}$-trees provides
an example of a polynomial $R(t)$ satisfying $(\cR)$, for which $\Psi$ is differentially transcendental.
So let us prove that $R(t)=t^2+t^3$ is not conjugated
via a homography
to a monomial or a Chebyshev polynomial.
\par
The set of critical points of $R(t)$ is the set of points in which the derivative $R'(t)$ annihilates,
namely, $C:=\{0,-\frac{2}{3}\}$.
The set of critical values of $R$ is the set of values of $R$ at the points of $C$,
i.e. $V:=\{0, \frac{4}{27}\}$. One says
that $R$ is post-critically finite if the set
    \[
    P(R):=\mathop{\cup}_{n\geq 0}R^{\circ n}(V)
    \]
is finite, where we have denoted by $R^{\circ n}$ the composition of $R$ with itself $n$ times.
The property of being post-critically finite is invariant under conjugation with a
homography $m(t)$, since homographies do not have critical points and
$m\circ R^{\circ n}\circ m^{-1}=(m\circ R\circ m^{-1})^{\circ n}$.
It is easy to check that $P(t^d)$ is finite.
For the finiteness of $P(\pm T_d(t))$ we refer to~\cite[Problem 7-c, page 73]{milnor_dynamics_2006}.
\par
Since for any $t_0\in\left(-\frac{1+\sqrt{5}}{2},-\frac{1-\sqrt{5}}{2}\right)$ we have $|t_0+t_0^2|<1$, we conclude that :
    \[
    |t_0^2+t_0^3|<|t_0|.
    \]
The fact that $\frac{4}{27}\in \left(-\frac{1+\sqrt{5}}{2},-\frac{1-\sqrt{5}}{2}\right)$
implies that
the set $P(R)$ is infinite with an accumulation point at $0$.
It follows from Theorem~\ref{thm:Becker-Bergweiler} from \cite{becker_hypertranscendency_1995}, that
in this case $\Psi$ is differentially transcendental.
\end{exa}

\begin{rmk}\label{rmk:ritt-missed-case}
Notice that if $R'(0)$ is not $0$ or a root of unity, one easily shows that $\Psi\in t+t^2C[[t]]$ substituting a formal power series in the functional equation and showing that one can determine its coefficients recursively (see Lemma~\ref{lemma:iterativelogarithm}, below).
In this case, there exists a power series $\tau$ such that $\phir(\tau)=R'(0)\,\tau$ and $\Psi=\frac{\tau}{\tau'}$.
As far as the differential transcendence of $\tau$ is concerned, Ritt in \cite{ritt_transcendental_1926} has classified the cases in which
$\tau$ is differentially transcendental. Differently from the Abel and the Böttcher cases,
$\tau$ can be both transcendental and differentially algebraic over $\C(t)$, which would add an extra step to the proof of our main results (see \S\ref{sec:mainresults-1}).
\end{rmk}

For further reference, we notice
Equation~\eqref{eq:Julia} can be derived to obtain a system of functional equations satisfied by $\Psi$ and its derivatives.
Before describing it, we introduce the notation:

\begin{notation}
For $\Psi$ and $\Psi$ only, we will write $\Psi_n$ for $\Psi^{(n)}$, for any integer $n\geq 1$, and $\Psi_0$ for $\Psi$ itself.
\end{notation}

\begin{lemma}\label{lemma:derivations-psi}
The solution $\Psi$ and its derivatives satisfy the following system of functional equations:
        \[
        \l\{\begin{array}{ll}
             \phir(\Psi_0)=R'\Psi_0,&  \\
             \phir(\Psi_n)=(R')^{1-n}\Psi_n+\sum_{k=0}^{n-1}A_{n,k}\Psi_k,& \forall n\geq 1,
        \end{array}\r.
        \]
where $A_{n,k}$ is a rational expression in $R',\dots,R^{(n-k+1)}$ with coefficients in $\C$.
\end{lemma}

\begin{proof}
Notice that the chain rule holds in $\C((t))$ as well as in $\C((t))(\log t)$.
Differentiating $\phir(\Psi_0)=R'\Psi_0$ we obtain the functional equation for $\Psi_1$,
namely, $\phir(\Psi_1)=\Psi_1+\frac{R''}{R'}\Psi_0$.
We can apply the chain rule to prove the statement for $\Psi_n$ by induction on $n$.
\end{proof}

\section{Proofs of Theorem~\ref{thmINTRO:a=1}
and Theorem~\ref{thmINTRO:b=0}}
\label{sec:mainresults-1}

The main workhorse results  of this section are
Theorem~\ref{thmINTRO:main-additive} and Theorem~\ref{thmINTRO:main-multiplicative},
proved in \S\ref{sec:main-additive} and \S\ref{sec:main-homogenous}, respectively.
To deduce Theorem~\ref{thmINTRO:a=1}
and Theorem~\ref{thmINTRO:b=0} from the latter, we rely on the main results of \cite{becker_hypertranscendency_1995}.

\subsection{Proof strategy}
\label{sec:proof-strategy-rational}

We recall Theorem~\ref{thmINTRO:a=1}  from the introduction:
\mainA*

We can say something more precise under some additional conditions. For example, in the application to the generating functions of complete trees, we use the following corollary (see Corollary~\ref{corINTRO:trees-chainsaw} in the introduction):

\begin{cor}\label{cor:trees-chainsaw}
Let $R\in t^2\C[t]\setminus\{0\}$ and $b\in t\C[t]$, with $b\neq 0$ and $\deg_tb<\deg_tR$.
If there exists $f\in\C((t))$ such that $\phir(f)=f+b$, then $f$ is differentially transcendental over $\C(t)$.
\end{cor}

\begin{proof}
Let us suppose towards a contradiction that~$f$ is differentially algebraic over $\C(t)$.
Theorem~\ref{thmINTRO:a=1}, plus the fact that $b\neq 0$, implies that $f$ is a nonconstant rational function.
We can write $f=\frac{U}{V}$, for some $U,V\in \C[t]$, such that $\gcd(U,V)=1$. Let us assume that $V$ is not a constant in $\C$.
Then the functional equation $\phir(f)=f+b$ becomes
$\frac{\phir(U)}{\phir(V)}=\frac{U+bV}{V}$, with $\gcd(U+bV,V)=1$.
Since $R$ is a polynomial, both $\phir(U)$ and $\phir(V)$ are polynomials.
Moreover there exist $u,v\in\C[t]$ such that $uU+vV=1$, hence we conclude that
$\phir(u)\phir(U)+\phir(v)\phir(V)=1$, i.e., $\gcd(\phir(U),\phir(V))=1$.
Therefore both
$\frac{\phir(U)}{\phir(V)}$ and $\frac{U+bV}{V}$ are irreducible rational functions in $\C(t)$ and there must exist $c\in \C$ such that $\phir(V)=cV$.
Since $V$ is not a constant polynomial, we obtain the identity $\deg_t V\cdot\deg_t R=\deg_t V$, hence
$\deg_tR=1$, which contradicts the fact that $R\in t^2\C[t]$. We conclude that $V$ is a constant and
$f\in \C[t]$.
\par
We deduce from
$\phir(f)=f+b$ that $\deg_tf\cdot\deg_t R\leq\max(\deg_tf,\deg_tb)\leq \max(\deg_tf,\deg_tR)$ and hence
that either $\deg_tR=1$ or $\deg_tf=1$.
Since $R$ has at least degree $2$, $f$ must have degree $1$ in $t$.
We write $f=\alpha t+\beta$, for some $\alpha,\beta\in\C$. Then
we have
$\alpha R(t)+\beta=\alpha t+\beta +b$ and hence $R(t)=t+\frac{b}{\alpha}$. We conclude that
$\deg_tR\leq\deg_tb$, counter to the assumption.
We have found a contradiction, therefore $f$ is necessarily differentially transcendental over $\C(t)$.
\end{proof}

Finally we recall the statement of Theorem~\ref{thmINTRO:b=0}:
\mainB*


The next proposition provides insight on the different conclusions of the two statements above,
namely that an equation of the form $\phir(y)=y+b$ cannot have algebraic non-rational solutions.
It could be proved as an immediate consequence of the fundamental theorems of difference Galois theory, but one can also give a relatively short and elementary proof:

\begin{prop}\label{prop:algebraicity}
Let $R\in \C(t)$ be a nonzero rational function having a zero at $0$.
Suppose that there exists $b\in \C(t)$ and~$f\in\C((t))$ such that $\phir(f)=f+b$.
Then either $f\in \C(t)$ or $f$ is transcendental over $\C(t)$.
\end{prop}

The proof is based on the following crucial property of $\phir$, which for further reference,
we state for any $R\in t\C[[t]]$, not necessarily rational:

\begin{lemma}\label{lemma:constants-Phi-R}
Let $R=\sum_{n\geq 1}r_nt^n\in t \C[[t]]$, $R\neq 0$ such that no iteration of $R(t)$ is equal to $t$.
The field $\C((t))^{\phir}$
of invariant elements of $\C((t))$ with respect to $\phir$ coincides with $\C$.
\end{lemma}

\begin{proof}
We suppose by contradiction that there exists
$f=\sum_nf_nx^n\in\C((t))\smallsetminus \C$ such that $\phir(f)=f$.
Since either $f\in \C[[t]]$ or $\frac{1}{f}\in \C[[t]]$, we can suppose
without loss of generality that $f\in \C[[t]]$.
Let $f_0$ be the constant term of $f$. Since $\phir(f-f_0)=f-f_0$,
we can also suppose that $f_0=0$, i.e. $f$ is a nonzero element of $t\C[[t]]$.
Let $f_Nt^N$ and $r_Mt^M$ be the monomials of
lowest degree in $t$
of $f$ and $R$, respectively, so that $f_N,r_M\neq 0$.
We derive from $\phir(f)=f$ that $f_Nr_M^Nt^{NM}=f_Nt^N$.
We conclude that we must have $M=1$, $r_1^N=1$ and $f_N\in\C$.
By replacing $R$ by its $N$-fold iterate, we can suppose that $r_1=1$.
Finally, the development of $\phir(f)=f$ gives the identity:
    \[
    \sum_{n\geq N}\l(\sum_{h=N}^n f_h
    \sum_{m_1+\dots+m_h=n}r_{m_1}\cdots r_{m_h}\r)  t^n
    =\sum_{n\geq N}f_nt^n.
    \]
Identifying the coefficients of $t^N$ on both side of the identity above,
we obtain once more the identity $f_Nr_1^N=f_N$.
Calculating the coefficients of $t^{N+1}$,
we obtain $Nr_1^{N-1}r_2f_{N}+r_1^{N+1}f_{N+1}=f_{N+1}$, hence $Nr_2f_{N}+f_{N+1}=f_{N+1}$ and $r_2=0$.
Let $P$ be the smallest integer~$\geq 2$ such that $r_P\neq 0$.
Identifying the coefficients of $t^{N+P-1}$ in the equation above, we obtain that
$Nr_1^{N-1}r_Pf_N+r_1^{N+P-1}f_{N+P-1}=f_{N+P-1}$. Therefore
$Nr_Pf_N+f_{N+P-1}=f_{N+P-1}$, and hence $f_N=0$ contrary to the assumption.
Therefore the only invariant elements with respect to $\phir$ are the elements of $\C$.
\end{proof}

\begin{proof}[Proof of Proposition~\ref{prop:algebraicity}.]
It is enough to prove that if $f$ is algebraic over $\C(t)$, then it is a rational function.
If $f$ is algebraic over $\C(t)$, there exists a polynomial $P(t,X)\in \C(t)[X]$ of minimal
degree in $X$ such that $P(t,f)=0$. We can suppose that $P$ is monic of degree $N$,
i.e., $P(t,X)=\sum_{i=0}^{N-1}p_i(t)X^i+X^N$, for some integer $N\geq 1$ and
$p_0(t),\dots,p_{N-1}(t)\in \C(t)$.
Applying $\phir$ to $P(t,f)=0$ we find another monic polynomial of degree~$N$ in~$\C(t)[X]$ having $f$ as a root, namely $P(R(t),X+b)=\sum_{i=0}^{N-1}p_i(R(t))(X+b)^i+(X+b)^N$.
Since $P(t,X)$ and $P(R(t),X+b)$ are both monic minimal polynomials of $f$ over $\C(t)$ they must coincide.
In particular they must have the same coefficients of $X^{N-1}$, forcing $p_{N-1}(t)=p_{N-1}(R(t))+Nb$, that is, $\phir(p_{N-1}(t))=p_{N-1}(t)-Nb$.
Now consider the image of $Nf+p_{N-1}(t)$ under $\phir$:
    \[
    \phir(Nf+p_{N-1}(t))=N\phir(f)+\phir(p_{N-1}(t))=Nf+Nb+p_{N-1}(t)-Nb=Nf+p_{N-1}(t).
    \]
Hence $Nf+p_{N-1}(t)$ is invariant with respect to $\phir$ and so Lemma~\ref{lemma:constants-Phi-R} implies that $Nf+p_{N-1}\in \C$. By solving for $f$ we conclude $f\in \C(t)$.
\end{proof}

In general, the same is not true for homogeneous equations. For example,  the irrational, but algebraic, function $f(t) = \frac{1}{\sqrt{1 - 4t^2}}$ satisfies $f\left(\frac{t^2}{1-2t^2}\right)=(1-2t^2)f(t)$.

Next we give some more detail on the strategy of the proof of Theorem~\ref{thmINTRO:a=1} and Theorem~\ref{thmINTRO:b=0}.

\begin{proof}[Idea of the proof of Theorem~\ref{thmINTRO:a=1}]
Let us suppose that $R(t)$ has a zero at $0$ of order $d\geq 2$ and
let~$\tau$ be the associated solution of $\phir(\tau)=\tau^d$ and~$\sigma$ its compositional inverse.
By Theorem~\ref{thm:Becker-Bergweiler} and Remark~\ref{rmk:Beker-Bergweiler},
if $R$ is conjugate via a homography either to~$t^d$ or
to~$\pm T_d(t)$, then $\sigma(t)\in\C(t)$ is obtained composing either $\zeta t$
or $\frac{1}{2}\left(\zeta s+\frac{1}{\zeta s}\right)$ with a homography and hence it is a rational function.
In this case,~$\tau$ is at worse algebraic.
We consider the functional equation $\phir(f)=f+b$ and
we set $w=f\circ\sigma$, $g=b\circ\sigma$ and $s=\tau(t)$. Then $\phir(f)=f+b$ becomes
$f(\sigma(\tau(R(t))))=f(\sigma(\tau(t)))+b(\sigma(\tau(t)))$ and hence
$w(s^d)=w(s)+g(s)$, with $g\in\C(s)$.
Then we know that $w(s)=f(\sigma(s))$ is either a rational function or differentially transcendental over $\C(s)$ (see \cite[page~22]{rande_mahler_1992}).
Since $\tau(t)$ is an algebraic function, $f(t)=w(\tau(t))$  is either an algebraic function of $t$ or differentially transcendental over the field of algebraic functions, and \emph{a fortiori} differentially transcendental over $\C(t)$.
Thanks to Proposition~\ref{prop:algebraicity}, we conclude that $f$ is either rational or differentially transcendental  over $\C(t)$.
\par
It remains to prove the theorem when $\tau$ is differentially transcendental over $\C(t)$,
which furthermore implies that $\Psi$ of the previous section is
differentially transcendental (see Proposition~\ref{prop:Psi}).
This case is handled in Theorem~\ref{thmINTRO:main-additive},  whose proof is in \S\ref{sec:main-additive}.
\end{proof}

\begin{proof}[Idea of the proof of Theorem~\ref{thmINTRO:b=0}]
The same argument as above can be used to deduce Theorem~\ref{thmINTRO:b=0}
from Theorem~\ref{thmINTRO:main-multiplicative} (see \S\ref{sec:main-homogenous} below for the proof).
\change{Let us give some details.
The same change of variable and unknown
function that we have considered in the proof of Theorem~\ref{thmINTRO:a=1} allows to prove that, when $\tau$ is differentially algebraic, $f$ is either an algebraic function or differentially transcendental. To conclude, we need to show that there exists a positive integer $N$ such that $f^N$ is rational. We proceed as in the proof of Proposition~\ref{prop:algebraicity}.
Let us consider the minimal monic polynomial of $f$ over $\C(t)$, so that
$f^N+\sum_{i=0}^{N-1}p_i(t)f^i=0$, with $p_i(t)\in\C(t)$ and $N$ minimal.
Applying $\phir$ and dividing by $a^N$ we obtain:
    \[
    0=f^N+\sum_{i=0}^{N-1}\frac{p_i(t)}{a^{N-i}}f^i.
    \]
The two polynomials must coincide. In particular, we must have $\phir(p_0)=a^Np_0$. Since $f^N$ also satisfies
$\phir(f^N)=a^Nf^N$, we conclude that there must exists a constant $c\in\C$ such that $f^N=cp_0(t)\in\C(t)$, as claimed.}
\end{proof}

In the proofs above,
we reduce to the case of Mahler equation whenever $R'(0)=0$ and $\Psi$ is differentially algebraic over $\C(t)$. In practice, such a condition may be difficult to check, so it is nice not have to. However, one may ask whether
the case of a differentially transcendental $\Psi$ actually occurs,
or if the known theorems on Mahler equations alone would suffice to
conclude differential transcendency in the  applications.
We know from Example~\ref{exa:Psi} that the answer is no:

\begin{prop}\label{prop:novelty}
The functional equation of the complete $(2,3)$-tree, with $R(t)=t^2+t^3$,
cannot be reduced to a Mahler equation via a rational conjugation.
\end{prop}

\begin{proof}
We have shown in Example~\ref{exa:Psi} that $R(t)$ cannot be conjugate to a monomial or
a Chebyshev polynomial, via a homography. Therefore
Theorem~\ref{thm:Becker-Bergweiler} implies that Equation~\ref{eq:23treeeqn} cannot be
reduced to a Mahler equation via a rational variable change.
\end{proof}

\subsection{\texorpdfstring{Iterative difference equations of the form $\phir(y)=y+b$}{y(R)=y+b}. Proof of Theorem~\ref{thmINTRO:main-additive}}
\label{sec:main-additive}

We fix a \emph{nonzero} $R\in\C(t)$, with $R(0)=0$, with no further assumptions on $R'(0)$.
We denote by $\F$ the field generated by $\C((t))$ and $\Psi$ and all its derivatives with respect to $\frac{d}{dt}$.
Until the end of the paper we suppose the following assumptions to be verified:

\begin{quote}\mbox{}
\begin{minipage}{.1\textwidth}$(\cDT)$\end{minipage}
\begin{minipage}{.7\textwidth}
We fix $R\in\C(t)$, with $R(0)=0$, such that no iteration of $R$ is the identity.
We suppose that there exists $\Psi\in\F$ such that
$\phir(\Psi)=R'(t)\Psi$ and that $\Psi$ is differentially transcendental over~$\C(t)$ with respect to $\frac{d}{dt}$.
\end{minipage}
\end{quote}
This means that either $\F=\C((t))$ if $R'(0)\neq 0$ or $\F=\C((t))(\log t)$, if $R'(0)=0$ (see Proposition~\ref{prop:Psi}).
We have seen in \S\ref{sec:iterative-log} that the field $\F$ comes equipped with an extension of the endomorphism~$\phir$ and an extension of the derivative $\frac{d}{dt}$. Moreover, the sub-field of $\F$ of constants is $\C$, which is also the sub-field of $\F$ of invariant elements with respect to $\phir$.
We recall Theorem~\ref{thmINTRO:main-additive}, whose proof is the purpose of this section:

\mainD*


The proof of Theorem~\ref{thmINTRO:main-additive} is based on a result of difference Galois theory,
that we will recall shortly. The following lemma explains the crucial role of $\Psi$ in the proofs below:

\begin{lemma}\label{lemma:CommutationPartialPhi}
The derivation $\partial:=\Psi\frac{d}{dt}$ commutes with $\phir$, namely
$\partial\circ\phir=\phir\circ\partial$.
\end{lemma}

\begin{proof}
For any $f\in\F$ we have:
$(\partial\circ\phir)(f)=\Psi\phir(f')R'=\phir(f')\phir(\Psi)=\phir(f'\Psi)=(\phir\circ\partial)(f)$.
\end{proof}

We call $\K$ the field generated by $\C(t)$, the function $\Psi$ and its derivatives with respect to $d/dt$.
We are exactly in the situation considered in \cite{di_vizio_difference_2021}, namely:
we have a base field $\K$ contained in a larger field $\F$;
the endomorphism $\phir$ acts on both $\K$ and $\F$ and in both cases its subfield
of constants is $\C$;
the derivative $\partial:=\Psi\frac{d}{dt}$ stabilizes $\K$ in $\F$ and commutes to $\phir$.
Therefore we can study the differential transcendence of $f$ using \cite[Corollary 2.6.7]{di_vizio_difference_2021},
which is a restatement of \cite{hardouin_hypertranscendance_2008}:

\begin{prop}\label{prop:DiffAlgHanoi}
Let $\F$ be a field equipped with an endomorphism $\Phi_R$ and a derivation $\partial$, such that
$\Phi\circ\partial=\partial\circ\Phi$ and let $\K$ be a sub-field of $\F$
stable by $\Phi$ and $\partial$ and containing the field $C$ of invariant elements of $\F$
with respect to $\Phi$.
Let $b\in\K$, $b\neq 0$, and $f\in\F$ be  a solution of $\phir(y)=y+b$. The following statements are equivalent:
\begin{enumerate}
  \item There exist $n\geq 0$, $\lambda_0,\dots,\lambda_n\in  C$, not all zero, and $g\in\K$ such that $\lambda_0b+\lambda_1\partial(b)+\dots+\lambda_n \partial^n(b)=\phir(g)-g$.
  \item There exist $n\geq 0$, $\lambda_0,\dots,\lambda_n\in  C$, not all zero, such that $\lambda_0f+\lambda_1\partial(f)+\dots+\lambda_n \partial^n(f)\in\K$.
  \item $f$ satisfies a linear differential equation in $\partial$ with coefficients in
  $\K$, i.e., $f$ is $\partial$-finite over $\K$.
  \item $f$ satisfies an algebraic differential equation in $\partial$ with coefficients in
  $\K$, i.e., $f$ is differentially algebraic over $\K$.
\end{enumerate}
\end{prop}

\begin{rmk}\label{rmk:DiffAlgHanoi}
We are not proving the proposition above, but we can quickly comment on the equivalence between the first three assertions.
Since $ \Phi$ and $\partial$ commute, if $f$ satisfies a differential equation of order $n$ over $\K$, the system of functional
equations
    \begin{align*}
      &  \Phi(y)=y+b \\
      &  \Phi(\partial(y))= \partial(y)+\partial(b)\\
      & \vdots \\
      &  \Phi(\partial^n(y))= \partial^n(y)+\partial^n(b)
    \end{align*}
has a set of solutions $f,\partial(f),\dots,\partial^n(f)$, which is algebraically dependent over $\K$.
For any $\lambda_0,\dots,\lambda_n\in  C$, by the $\mathbb{C}$-linearity of $ \Phi$,
the element $\widetilde f=\lambda_0f+\lambda_1\partial(f)+\dots+\lambda_n \partial^n(f)$ of $\F$ verifies the functional equation
$ \Phi(\widetilde f)=\widetilde f+\lambda_0b+\lambda_1\partial(b)+\dots+\lambda_n \partial^n(b)$.
If $g\in\K$ is as in the first assertion above, then $\widetilde{f}$ and $g$ satisfy the
same functional equation, therefore $\widetilde{f}-g\in C$.
Summarizing, $g\in\K$ if and only if $\widetilde{f}\in\K$ if and only if $f$ is solution of a linear differential equation in $\K$.
The strength of this result is therefore the equivalence between (each one of) the first
three assertions and the fourth one, which requires some Galoisian insight.
\end{rmk}

Before proving Theorem~\ref{thmINTRO:main-additive}, we prove the following lemma,
in which we use the definitions and properties stated in \S\ref{subsec:DiffAlgebra}:

\begin{lemma}\label{lemma:derivatives}
We assume, as above, that $\Psi$ is differentially transcendental over $\C(t)$.
Let $L\subset\C((t))$ be a differentially algebraic extension of $\C(t)$.
Then for any $f\in L$ and any integer $n\geq 2$, we have $\partial^{n}f\in f'\Psi_0^{n-1}\Psi_{n-1}+L[\Psi_0,\ldots,\Psi_{n-2}]$.
\end{lemma}

\begin{proof}
First of all, Lemma~\ref{lemma:DiffAlg} implies
that $\Psi$ is differentially transcendental over $L$.
As far as the first statement is concerned, by definition, $\partial f=f'\Psi_0$, with $f'\in L$.
For $n=2$, we have:
$\partial^{2}f=f''\Psi_0^2 +f'\Psi_{1}\Psi_0 $. The general statement
is proved by induction on $n$ using the Leibniz formula,
since
$\partial(f'\Psi_0^{n-1}\Psi_{n-1})=f'\Psi_0^n\Psi_n+f'\partial(\Psi_0^{n-1})\Psi_{n-1}
+f''\Psi_0^n\Psi_{n-1}$.
\end{proof}

\begin{proof}[Proof of Theorem~\ref{thmINTRO:main-additive}]
It is enough to prove that if $f\in\C((t))$ is differentially algebraic over $\C(t)$, then $f\in\C(t)$.
The proof is divided in two steps: first we prove that $f'\in\C(t)$ and then we deduce that $f\in\C(t)$.
\par
We start with some elementary but crucial remarks.
If $b=0$, then $f\in\C\subset\C(t)$, therefore we can suppose $b\neq 0$.
If $f$ is differentially algebraic over $\C(t)$ then $f$ is differentially algebraic over $\K$, with respect to $\partial$.
The commutativity of $\partial$ and $\phir$ implies that for any positive integer $n$ we have
$\phir(\partial^n(f))=\partial^n (f)+\partial^n(b)$. Moreover, the definition of $\phir$ plus the chain rule imply that
    \begin{equation}\label{eq:z'-functeq}
    \phir(f')R'=f'+b'.
    \end{equation}
Since $b\neq 0$, $f$ is not a constant and hence $f'\neq 0$.
\par
It follows from the second assertion of Proposition~\ref{prop:DiffAlgHanoi} that there exist
an integer $n\geq 0$, $\lambda_0,\dots,\lambda_n\in\C$, not all zero, and
$g\in\K$ such that:
\begin{equation}\label{eq:telescoper-1}
    \sum_{i=0}^n\lambda_i\partial^i(f)=g.
\end{equation}
We can suppose that $\lambda_n\neq 0$.
If $n=0$, we have $f=\frac{g}{\lambda_0}$,
therefore $f\in\K$. By assumption, $\K/\C(t)$ is a purely differentially transcendental extension
and $f$ is differentially algebraic over $\C(t)$, hence we conclude that $f\in\C(t)$ (see Lemma~\ref{lemma:DiffAlg}) .
This proves the whole theorem for $n=0$, hence we suppose from now on that $n\geq 1$.
\par
We are now ready to actually prove the theorem.

\begin{step}
If $f\in\C((t))$ is differentially algebraic over $\C(t)$ then $f'\in\C(t)$.
\end{step}

Let $L$ be the differentially algebraic extension of $\C(t)$ generated by $f$ and its derivatives.
Notice that Lemma~\ref{lemma:DiffAlg} implies that $\Psi$ is differentially transcendental over $L$.
Lemma~\ref{lemma:derivatives} implies that \[g\in\K\cap\left(\lambda_n f'\Psi_0^{n-1}\Psi_{n-1}+L[\Psi_0,\dots,\Psi_{n-2}]\right).\]
Since the $\Psi_i$'s are algebraically independent both on $\C(t)$ and $L$, we can apply $\frac{\partial}{\partial\Psi_0}$
and $\frac{\partial}{\partial\Psi_{n-1}}$ to $g$ and we obtain:
    \[
    \lambda_n f'=\frac{1}{(n-1)!}\left(\frac{\partial}{\partial\Psi_0}\right)^{n-1}
    \left(\frac{\partial}{\partial\Psi_{n-1}}\right)(g)\in\K.
    \]
Since $\lambda_n\neq 0$,
we have proved that $f'$, in addition to being differentially algebraic over $\C(t)$, belongs to $\K$.
We conclude thanks to Lemma~\ref{lemma:DiffAlg} that $f'\in\C(t)$.

\begin{step}
If $f'\in\C(t)$, then actually $f\in\C(t)$.
\end{step}

Since $f'\in\C(t)$, then $L=\C(t)(f)$ and there exist an integer
$I\geq 0$, $c_1,\dots,c_I,a_1,\dots,a_I\in\C$ and $w\in\C(t)$, such that:
    \[
    f'=\frac{c_1}{t-a_1}+\dots+\frac{c_I}{t-a_I}+w'.
    \]
If we set $s_i=\frac{R(t)-a_i}{t-a_i}$, then
    \[
    0=\phir(f')R'(t)-f'-b'=\sum_{i=1}^Ic_i\frac{s_j'}{s_j}+\phir(w')R'(t)-w'-b'.
    \]
The summands of $\sum_{i=1}^Ic_i\frac{s_j'}{s_j}$ are logarithmic derivatives of rational functions
and $\phir(w')R'(t)-w'-b'$ has zero residue at any pole,
because it admits a primitive in $\C(t)$.
Therefore these two terms are both equal to zero. In particular, $\phir(w')R'(t)=w'+b'$.
We deduce that $\phir(f'-w')R'(t)=f'-w'$, and finally that $\phir((f'-w')\Psi_0)=(f'-w')\Psi_0$.
Since $\F^{\phir}=\C$, there exists $c\in\C$ such that
$f'-w'=\frac{c}{\Psi_0}$. If $c\neq 0$ we obtain a contradiction since $f'-w'\in\C(t)$, while $\Psi_0$ is differentially transcendental over $\C(t)$.
Hence $c=0$ and $f'=w'$. The latter equality implies that $f\in\C(t)$, since by construction $w\in\C(t)$.
\end{proof}

\subsection{\texorpdfstring{Iterative difference equations of the form $\phir(y)=ay$}{y(R)=ay}.
Proof of Theorem~\ref{thmINTRO:main-multiplicative}}
\label{sec:main-homogenous}

Let us consider a functional equation $\phir(y)=ay$, with $a\in\C(t)$.
We recall the statement of Theorem~\ref{thmINTRO:main-multiplicative}, whose proof is the object of this
subsection:

\mainE*


We start with a lemma for
functional equations of order $1$ with constant coefficients, which partially explains the conclusion of the theorem
above.

\begin{lemma}\label{lemma:mainthm-multiplicative}
Under the assumptions $(\cDT)$, in the notation of Theorem~\ref{thmINTRO:main-multiplicative}, let $a\in\C$, $a\neq 0$.
Then either $a=1$ and $z\in\C$ or
$z$ is differentially transcendental over $\C(t)$ and there exists a nonzero $c\in\C$ such that
$\frac{z'}{z}=\frac{c}{\Psi}$.
\end{lemma}

\begin{proof}
For $a=1$ the statement is trivial, so let us suppose that $a\neq 1$.
Differentiating $\phir(z)=az$, one sees that $\Psi z'$ is solution of the same functional equation, i.e. $\phir(\Psi z')=a(\Psi z')$.
Using the fact that $\F^{\phir}=\C$, we deduce that there exists $c\in\C$ such that
$\Psi z'=cz$.
If $z$ is differentially algebraic over $\C(t)$,
then $\frac{z'}{z}=c\Psi^{-1}\in\K$ is also differentially algebraic over $\C$.
By assumption, $\Psi$ is differentially transcendental over $\C(t)$, we deduce that $c=0$,
hence $z$ is a constant
in contradiction with the fact that $a\neq 1$.
\end{proof}


\begin{proof}[Proof of Theorem~\ref{thmINTRO:main-multiplicative}]
From the lemma above, if $a=1$, then $z\in\C\subset\C(t)$.
If $a=0$, then $f=0\in\C(t)$.
Moreover, if $a\in\C\setminus\{0,1\}$ then $z$ is differentially algebraic
over $\K$, by Lemma~\ref{lemma:mainthm-multiplicative}.
Therefore let us suppose that $a$ is not a constant and that $z$ is differentially algebraic over $\C(t)$, and hence over $\K$.
By taking the logarithmic derivative we find that
    \[
    \phir\l(\frac{\partial z}{z}\r)=\frac{\partial z}{z}+\frac{\partial a}{a},
    \hbox{~or equivalently that~}
    \phir\l(\frac{z'}{z}\r)R(t)'=\frac{z'}{z}+\frac{a'}{a}.
    \]
Since $a$ is not a constant, the logarithmic derivative $\frac{a'}{a}$ is nonzero, hence
$\frac{z'}{z}\neq 0$.
The proof follows the structure of the proof of Theorem~\ref{thmINTRO:main-additive} replacing $f$ with $\frac{z'}{z}$.
\setcounter{step}{0}
\begin{step}
If $z$ is differentially algebraic over $\K$ then $\frac{z'}{z}\in\C(t)$.
\end{step}

Notice that $\frac{\partial z}{z}$ is differentially algebraic over $\K$
 with respect to $\partial$.
It follows from Proposition~\ref{prop:DiffAlgHanoi}
that there exist an integer $n\geq 0$, $\lambda_0,\dots,\lambda_n\in\C$, not all zero, and $g\in\K$ such that:
    \begin{equation}\label{eq:identity-z'/z}
    \lambda_0\frac{\partial z}{z}+\dots
    +\lambda_n\partial^n\l(\frac{\partial z}{z}\r)=g\in\K\,.
    \end{equation}
We can suppose that $\lambda_n\neq 0$.
If $n=0$ then $\frac{z'}{z}=\frac{g}{\lambda_0\Psi_0}$, therefore
$\frac{g}{\lambda_0\Psi_0}\in\K$ is differentially algebraic over $\C(t)$.
Since $\K/\C(t)$ is a purely differential extension, we conclude that
$\frac{g}{\lambda_0\Psi_0}\in\C(t)$, and hence $\frac{z'}{z}\in\C(t)$.
\par
Let us consider the case $n\geq 1$. Let $L$ be the differentially algebraic extension of $\C(t)$ generated by
$z$ and all its derivatives.
Then $\frac{\partial z}{z}=\Psi_0\frac{z'}{z}\in\Psi_0\,L$ and $\Psi_0$ is differentially transcendental over $L$
by Lemma~\ref{lemma:DiffAlg}.
\par
One proves recursively that $\partial^n\l(\frac{\partial z}{z}\r)\in \frac{z'}{z}\Psi_0^n\Psi_n+L[\Psi_0,\dots,\Psi_{n-1}]$,
since on one hand we have:
    \[
    \partial\l(\frac{\partial z}{z}\r)=\Psi_0\l(\Psi_0\frac{z'}{z}\r)'=\frac{z'}{z}\Psi_0\Psi_1+\l(\frac{z'}{z}\r)'\Psi_0^2,
    \]
and on the other hand, as in Lemma~\ref{lemma:derivatives}, we have:
    \[
    \partial\l(\frac{z'}{z}\Psi_0^{n}\Psi_n\r)=\frac{z'}{z}\Psi_0^{n+1}\Psi_{n+1}+\frac{z'}{z}n\Psi_0^n\Psi_1\Psi_n
    +\l(\frac{z'}{z}\r)'\Psi_0^{n+1}\Psi_n\,.
    \]
This proves in particular that the left hand side of
\eqref{eq:identity-z'/z} belongs to $\lambda_n\frac{z'}{z}\Psi_0^n\Psi_n+L[\Psi_0,\dots,\Psi_{n-1}]$, hence:
    \[
    \lambda_n \frac{z'}{z}=\frac{1}{n!}\left(\frac{\partial}{\partial\Psi_0}\right)^{n}
    \left(\frac{\partial}{\partial\Psi_{n}}\right)(g)\in\K.
    \]
Since $\lambda_n\neq 0$, we have proved that $\frac{z'}{z}$, in addition to being
differentially algebraic over $\C(t)$,
belongs to $\K$.
We conclude thanks to Lemma~\ref{lemma:DiffAlg} that $\frac{z'}{z}\in\C(t)$.

\begin{step}
If $\frac{z'}{z}\in\C(t)$ then $z$ is algebraic over $\C(t)$.
\end{step}

Notice that we have $\phir\l(\frac{z'}{z}\r)R'(t)=\frac{z'}{z}+\frac{a'}{a}$. We proceed exactly as in the proof of Theorem~\ref{thmINTRO:main-additive}.
Indeed, there exist an integer $I\geq 0$, $c_1,\dots,c_I,a_1,\dots,a_I\in\C$ and $w\in\C(t)$, such that:
    \[
    \frac{z'}{z}=\frac{c_1}{t-a_1}+\dots+\frac{c_I}{t-a_I}+w'.
    \]
If we set $s_i=\frac{R(t)-a_i}{t-a_i}$ then
    \[
    0=\phir\l(\frac{z'}{z}\r)R'(t)-\frac{z'}{z}-\frac{a'}{a}=
    \sum_{i=1}^Ic_i\frac{s_i'}{s_i}-\frac{a'}{a}+\phir(w')R'(t)-w'.
    \]
Since $\sum_{i=1}^Ic_i\frac{s_i'}{s_i}-\frac{a'}{a}$ is a sum of logarithmic derivatives
of rational functions and $\phir(w')R'(t)-w'$ has zero residue at any pole,
each one of these two terms must be zero, therefore $\phir(w')R'(t)=w'$.
This implies that $\phir(w'\Psi_0)=w'\Psi_0$ and hence that $w'\in\C(t)\cap\C\Psi_0^{-1}=\{0\}$, therefore
$\frac{z'}{z}=\frac{c_1}{t-a_1}+\dots+\frac{c_I}{t-a_I}$. Let us now consider the identity  $\sum_{i=1}^Ic_i\frac{s_i'}{s_i}=\frac{a'}{a}$.
We consider a basis $l_1,\dots,l_J$ of the $\Q$-vector space generated by $c_1,\dots,c_I$.
Therefore there exist $\lambda_i,j\in\Q$ such that
$c_i=\sum_{j=0}^J\lambda_{i,j}l_j$ so that
    \[
    \frac{a'}{a}=\sum_{i=1}^Ic_i\frac{s_i'}{s_i}=\sum_{j=1}^Jl_j\sum_{i=1}^I\lambda_{i,j}\frac{s_i'}{s_i}\,.
    \]
Since $a$ is not constant, $\frac{a'}{a}$ has at least a pole $t_0$ with a nonzero residue in $\Z$.
Calculating the residue at $t_0$ on the right hand side we find that there exists a non-trivial $\Q$-linear combination
of the $l_j$'s which belongs to $\Z$. This means that we can assume that $l_1=1$ and that $\Q$ is not contained in the
$\Q$-vector space generated by $l_2,\dots,l_J$.
In other words, $\sum_{j=2}^Jl_j\sum_{i=1}^I\lambda_{i,j}\frac{s_i'}{s_i}$ must have a rational residue at each pole,
but this cannot happen because $l_2,\dots,l_J$ are linearly independent over $\Q$.
We conclude that $\sum_{i=1}^I\lambda_{i,j}\frac{s_i'}{s_i}=0$ for any $j=2,\dots,J$,
as well as $\sum_{j=2}^Jl_j\sum_{i=1}^I\lambda_{i,j}\frac{s_i'}{s_i}=0$.
Hence $ \frac{a'}{a}=\sum_{i=1}^I\lambda_{i,1}\frac{s_i'}{s_i}$, with $\lambda_{i,1}\in\Q$.
Therefore there exists a constant $c\in\C$ such that $a=c\prod_{i=1}^I\l(\frac{R(t)-a_i}{t-a_i}\r)^{\lambda_{i,1}}\in\C(t)$.
Since $\prod_{i=1}^I(t-a_i)^{\lambda_{i,1}}$ is an algebraic function,
we can replace $z$ by $\widetilde{z}=z\prod_{i=1}^I(t-a_i)^{-\lambda_{i,1}}$.
Notice that $\widetilde z$ is differentially algebraic over
$\C(t)$ and satisfies $\phir(\widetilde z)=c\widetilde z$.
Therefore, Lemma~\ref{lemma:mainthm-multiplicative} implies that $\widetilde z$ is a constant and, hence, that $c=1$.
Therefore we have proved that $z=\widetilde z\prod_{i=1}^I(t-a_i)^{\lambda_{i,1}}$ is an algebraic function, as claimed.
\end{proof}

\section{Main result (and the algebraic case)}
\label{sec:Main}

The main purpose of this section is proving the first part of Theorem~\ref{thmINTRO:main-0},
whose statement we recall for the reader's  convenience:

\mainzero*


Like Theorem~\ref{thmINTRO:b=0} and Theorem~\ref{thmINTRO:a=1},
thanks to Proposition~\ref{prop:Psi},
we will deduce the theorem above from an analogous statement
with the extra assumption that $\Psi$ is differentially transcendental:

\begin{proof}[Idea of the proof of Theorem~\ref{thmINTRO:main-0}]
The proof is quite similar to the proofs of Theorem~\ref{thmINTRO:a=1} and Theorem~\ref{thmINTRO:b=0} (see \S\ref{sec:proof-strategy-rational}).
Using the same notation as in \S\ref{sec:proof-strategy-rational},
if $\Psi$ is differentially algebraic over $\C(t)$
there exists a rational function $\sigma$ such that $\sigma\circ\tau$ is the identity.
If we set $s:=\tau(t)$, $g:=f\circ\sigma$, $\alpha:=a\circ\sigma$ and $\beta=b\circ\sigma$, we find that
$g(s^d)=\alpha(s)g(s)+\beta(s)$, with $\alpha(s),\beta(s)\in\C(s)$.
We conclude once more
applying the known results on Mahler equations that
$f$ is either algebraic or differentially transcendental over $\C(t)$. To conclude the proof in this case we need to show that $f$ satisfies an homogeneous differential equation of order $1$ with rational coefficients. We proceed as we have already done in other proofs, namely we consider the minimal monic polynomial of $f$ over $\C(t)$, so that
$f^N+\sum_{i=0}^{N-1}p_i(t)f^i=0$. Applying $\phir$ and dividing by $a^N$ we obtain:
    \[
    0=\frac{(af+b)^N}{a^N}+\sum_{i=0}^{N-1}\phir(p_i(t))\frac{(af+b)^i}{a^N}=f^N+\left(N\frac{b}{a}+
    \frac{\phir(p_{N-1})}{a}\right)f^{N-1}+\dots,
    \]
which is still a minimal monic polynomial. Comparing the coefficients of $f^{N-1}$, we obtain that
$\phir(p_{N_1}(t))=ap_{N-1}(t)+Nb$, therefore
$p_{N-1}(t)$ and $Nf$ are solution of the same functional equation and $p_{N-1}-Nf$ is a solution of $\phir(y)=ay$. It follows from Theorem~\ref{thmINTRO:b=0}, that there exists $\alpha \in\C(t)$ such that $(p_{N-1}-Nf)'=\alpha(p_{N-1}-Nf)$,
which provides a differential equation for $f$ of the claimed form.
\par
Finally, we only need to prove
Theorem~\ref{thmINTRO:main-0} under the assumption $(\cDT)$, i.e., when
$\Psi$ is differentially transcendental over $\C(t)$, which is the main purpose of this section.
\end{proof}

\subsection{Algebraic setting}

In \S\ref{sec:mainresults-1}, we  studied functional equations whose space of solutions
is completely determined by a given solution in $\C((t))$ up to either a multiplicative or an additive constant.
This means that the space of solutions is contained in $\C((t))$, which is a field.
Although this is true in many applications, this is a major assumption,
since a general linear iterative functional equation does not have a
full set of solutions in a field.
In general, the solutions of $\phir(y)=ay$ have no reason to be Laurent series, as we have already seen with the solutions to the equation $\phir(y)=R'(t)y$.
As a matter of fact, they may even not be Puiseux series.
However, in order to describe the differential properties of $f$
we need to grasp some properties of the whole affine space of solutions of functional equations
of the form $\phir(y)=ay+b$.
For this reason we need to construct a $\K$-algebra equipped with an extension of $\phir$ and $\partial$ in which we can find the solutions that we need.
To do so, we will need a series of results from \cite{hardouin_differential_2008},
\cite{wibmer_existence_2012} and \cite{ovchinnikov_galois_2015}.
The theory is surveyed with similar notation and in view of similar applications in \cite[\S2]{bostan_differential_2020},
therefore we will refer to \emph{loc.cit.} instead of the original references.

\begin{notation}\label{notation:argebraic-R}
To be able to use the results in  \cite[\S2]{bostan_differential_2020}
we need to enlarge our base field, to ensure that $\phir$ is an isomorphism and hence
the existence of an abstract algebra containing ``enough''
solutions of our functional equations.
We consider the field $\bF$ of Puiseux series, i.e., the direct limit over $p\in\Z$, $p\geq 1$,
of the fields of Laurent series in $t^{1/p}$, with the natural inclusions.
The field $\bF$ is the algebraic closure of $\C((t))$. We call $\bK$ the algebraic closure of $\C(t)$ in $\bF$ and we fix a nonzero element $R=\sum_{n\geq 1}r_nt^{n/p}\in\bK$.
Since we can always consider a variable change $s=t^{1/p}$, without changing the nature of
our functional equation and the differential properties of the solutions, we will suppose
that $R\in\bK\cap\C[[t]]$: this simplifies  the presentation slightly since we won't have
to derive $t^{1/p}$ with respect to $t$ and therefore to carry some rational coefficients in the calculation.
Hence we suppose that $R=\sum_{n\geq 1}r_nt^{n}\in\bK$.
Let~$d$ be the order of zero of $R$ at $0$, so that $d$ is the smallest integer
such that $r_d\neq 0$.
Fixing a compatible system of roots of $R$ in $\bF$, we can extend
$\phir$ to an automorphism of $\bF$: Indeed $R$ admits a compositional inverse in $\bF$ even if $d>1$.
Lemma~\ref{lemma:constants-Phi-R} implies that $\bF^{\phir}=\C$, since for any $f\in\bF$ such that
$\phir(f)=f$ there exists a variable change of the form $\tilde s=t^{1/p}$, so that both $f$ and $R$ are in $\C((\tilde s))$.
Notice that the chain rule $\phir(f)'=\phir(f')R'$ still holds for any $f\in\bF$.
\end{notation}

\paragraph{Existence of $\Psi$.}
As before, we are interested in the functional equation
$y(R(t))=R'(t)y(t)$.
We start by showing the existence of a solution $\Psi$, which is a well-known fact:

\begin{lemma}\label{lemma:iterativelogarithm}
Let $R\in\bK\cap t\C[[t]]$.
The functional equation $\phir(y)=\frac{R'(t)}{d}y$ always admits a solution in $\C[[t]]$,
which is uniquely determined up to a multiplicative constant.
\end{lemma}

\begin{proof}
If $d>1$, i.e., $r_1=\dots=r_{d-1}=0$,
it follows from Böttcher's Theorem (See \cite[\S II]{ritt_iteration_1920} or \cite[\S8.3A]{kuczma_iterative_1990})
that there exists $\tau(t)\in t+t^2C[[t]]$, such that
$\tau(R(t))=\tau(t)^d$. Taking the logarithmic derivative, we see that
$\psi:=\frac{\tau(t)}{\tau'(t)}$ verifies the functional equation
$y(R(t))=\frac{R'(t)}{d}y(t)$.
\par
We consider now the case $d=1$, i.e. $r_1\neq 0$.
It follows from the functional equation $y(R(t))=R'(t)y(t)$ that, if  solution
$\psi=\sum_{n\geq 0} \psi_nt^n\in\C[[t]]$ exists,
then:
    \[
    \sum_{n\geq 1}\l(\sum_{h=1}^n\psi_h
    \sum_{k_1+\dots+k_h=n}r_{k_1}\cdots r_{k_h}\r)  t^n+\psi_0
    =
    \sum_{n\geq 1}
    \l(\sum_{h=0}^n(n-h+1)r_{n-h+1}\psi_h\r)t^n+r_1\psi_0.
    \]
As we have already pointed out, if $r_1$ is a root of unity,
replacing $R$ with a convenient iterate of $R$, we can assume that $r_1$ is equal to $1$.
Therefore we have to consider the case: $r_1=1$ and $r_1$ not a root of unity.
If $r_1$ is not a root of unity, then $\psi_0=0$ and we can choose any constant for $\psi_1$.
Recursively we see that
$(r_1^n-r_1)\psi_n$ is a linear combination of $\psi_1,\dots,\psi_{n-1}$.
That means that we can determine the coefficients of $\psi$ one by one, and hence the series $\psi$.
\par
Let us suppose that $r_1=1$ and that $r_2\neq 0$.
Identifying the coefficients of each $t^n$ in the expressions above, we obtain:
    \[
    \begin{array}{ll}
         \psi_0=\psi_0, & \hbox{for $n=0$, which is a tautology;}\\
         \psi_1=\psi_1+2r_2\psi_0, & \hbox{for $n=1$, hence $\psi_0=0$;}\\
         \psi_2+r_2\psi_1=\psi_2+2r_2\psi_1+3r_3\psi_0, & \hbox{for $n=2$, hence $\psi_1=0$;}\\
         \psi_3+2r_1r_2\psi_2+r_3\psi_1=\psi_3+2r_2\psi_2+3r_3\psi_1+4r_3\psi_0,
            & \hbox{for $n=3$, hence $\psi_2$ can be any constant;}
    \end{array}
    \]
and we keep solving inductively the system.
If $r_2=\dots,r_m=0$, with $m>2$, the calculation is similar and it is developed in full details,
for instance, in \cite[\S8.5A]{kuczma_iterative_1990}.
\end{proof}

In what follows we call $\Falg$ the field obtained in the following way: If $r_1\neq 0$ then $\Falg=\bF$, while $\Falg:=\bF(\log t)$ if $r_1=0$.
Notice that $\bF(\log t)$ can be defined as $\C((t))(\log t)$ in Lemma~\ref{lemma:construction-F},
looking at the elements of $\bF$ convergent in a punctured disk as Laurent series in some $t^{1/p}$ as
meromorphic functions on the Riemann surface of the logarithm.
We have $\Falg^{\phir}=\C$, thanks to the same theorems cited in the proof of Lemma~\ref{lemma:construction-F}.
The notation $\Falg$ mirrors the fact that
we have constructed $\Falg$ in the same way as $\F$, apart that we have replaced $\C((t))$ with
its algebraic closure. The same remark applies to $\Kalg$. Notice that they are both obtaining
adjoining the solution $\Psi$, constructed in the next corollary, and all its derivatives to $\bF$ and $\bK$, respectively.

\begin{cor}
The functional equation $\phir(y)=R'y$ admits a solution
$\Psi\in\Falg$.
Moreover $\Psi$
satisfies the system of functional equations described in Lemma~\ref{lemma:derivations-psi} and
the derivation $\partial:=\Psi\frac{d}{dt}$ verifies $\phir\circ\partial=\partial\circ\phir$.
\end{cor}

\begin{proof}
If $r_1\neq 0$, i.e. $d=1$, then we set $\Psi=\psi$. Otherwise, in the notation of the proof of the lemma above,
we have $\Psi=\psi\log\tau=\frac{\tau}{\tau'}\log\tau\in\Falg$.
The proofs of Lemma~\ref{lemma:derivations-psi} and Lemma~\ref{lemma:CommutationPartialPhi} are
a formal consequence of the chain rule, therefore they apply also to the present situation.
\end{proof}

We make the following assumptions:
\begin{quote}\mbox{}
    \begin{minipage}{.1\textwidth}$(\cDT_{alg})$\end{minipage}
    \begin{minipage}{.7\textwidth}
        We fix a nonzero $R\in\bK\cap t\C[[t]]$, such that no iteration of $R$ is equal to the identity.
        We suppose that $\Psi\in \Falg$ is differentially transcendental over $\bK$
        with respect to $\frac{d}{dt}$.
    \end{minipage}
\end{quote}
Notice that $\bK/\C(t)$ is an algebraic extension, therefore it is a differentially algebraic extension.
It follows from Lemma~\ref{lemma:DiffAlg} that being differentially transcendental over $\bK$
is equivalent to being differentially transcendental over $\C(t)$.

For further reference we prove the following rationality result, which generalizes Proposition~\ref{prop:algebraicity}:

\begin{lemma}\label{lemma:rat-2}
\change{Under the assumptions $(\cDT_{alg})$, if a functional equation $\phir(y)=ay+b$, with $a,R,b\in\C(t)$ and $a,b\neq 0$, has an algebraic solution $f\in\bK$, but the
associated homogenous equation $\phir(y)=ay$ has no nonzero algebraic solutions,
then $f$ is actually a rational function.}
\end{lemma}

\begin{rmk}
Let $R\in\C(t)$. Then $\phir(y)=t$, has the inverse of $R$ as solution, which in general is not rational.
This shows that the assumption $a\neq 0$ is necessary.
\end{rmk}

\begin{proof}
We use once more a classical trick: 
Let $\sum_{i=0}^Np_i(t)f^i=0$, with $p_i(t)\in\C(t)$ and $p_N=1$, be the minimal monic polynomial for an algebraic solution $f$ of
$\phir(y)=ay+b$. Applying $\phir$ to $\sum_{i=0}^Np_i(t)f^i=0$ and dividing by $a^N$, we obtain
    \[
    0=(a)^{-N}\sum_{i=0}^N\phir(p_i)\left(af+b\right)^i=f^N+\left(\frac{Nb+\phir(p_{N-1})}{a}\right)f^{N-1}+....,
    \]
which is also a minimal monic polynomial for $f$.
It follows that
    \[
    \phir(p_{N-1})=ap_{N-1}-Nb,
    \]
hence $N^{-1}p_{N-1}\in\C(t)$ and $f$ are solution of the same functional equation.
If $f\neq N^{-1}p_{N-1}$, then $f-N^{-1}p_{N-1}$ is an algebraic solutions of $\phir(y)=ay$, contradicting the assumptions.
Therefore $f=N^{-1}p_{N-1}$, i.e., $f$ is a rational function.
\end{proof}



\subsection{The particular cases \texorpdfstring{$a=1$}{a=1} and \texorpdfstring{$b=0$}{b=0} }
\label{subsec:a=1-b=0-algebraic}

We are going to start from the following partial result:

\begin{prop}\label{prop:main-additive-ALG}
Under the assumptions $(\cDT_{alg})$, we suppose that there exists $b\in\bK$ and $f\in\bF$ such that $f(R)=f+b$.
Then $f$ is either differentially transcendental over $\C(t)$ or the derivative $f'$ of $f$ is an algebraic function,
i.e., $f'\in\bK$.
\end{prop}



\begin{proof}
If $b=0$, then $f$ is a constant and there is nothing to prove. Therefore we suppose that $b\neq 0$ and
hence that $f$ is nonconstant, i.e., $f'\neq 0$.
Notice that if $f$ is differentially algebraic over $\bK$, \
then $f$ is differentially algebraic over $\Kalg$, with respect to $\partial$.
Moreover $\phir(f')R'=f'+b'$.
\par
It follows from Proposition~\ref{prop:DiffAlgHanoi} (applied to the field extension $\Falg/\Kalg$)
that there exist a non-negative integer $n$, $\lambda_{0},\ldots\lambda_{n}\in\C$, not all zero,
and $g\in\Kalg$ such that:
    \begin{equation}\label{eq:g_and_z}
    \sum_{i=0}^n\lambda_i\partial^i(f)=g.
    \end{equation}
We can suppose without loss of generality that $\lambda_n\neq 0$.
If $n=0$, we have $f=\frac{g}{\lambda_0}$,
therefore $f\in\Kalg$. By assumption, $\Kalg/\bK$ is a purely differentially transcendental extension
and $f$ is differentially algebraic over $\bK$, hence we conclude that $f\in L\cap\Kalg=\bK$ (see Lemma~\ref{lemma:DiffAlg}) .
This proves the proposition for $n=0$, hence we suppose from now on that $n\geq 1$.


Let $L$ be the differentially algebraic extension of $\bK$ generated by $f$ and its derivatives.
By Lemma~\ref{lemma:DiffAlg}, $\Psi$ is differentially transcendental over $L$.
Lemma~\ref{lemma:derivatives} implies that
\[g\in\Kalg\cap\left(\lambda_n f'\Psi_0^{n-1}\Psi_{n-1}+L[\Psi_0,\dots,\Psi_{n-2}]\right).\]
Since the $\Psi_i$'s are algebraically independent both on $\bK$ and $L$,
we can consider the partial derivatives $\frac{\partial}{\partial\Psi_0}$
and $\frac{\partial}{\partial\Psi_{n-1}}$ defined over $\Kalg$ and apply them to $g$.
We obtain:
    \[
    \lambda_n f'=\frac{1}{(n-1)!}\left(\frac{\partial}{\partial\Psi_0}\right)^{n-1}
    \left(\frac{\partial}{\partial\Psi_{n-1}}\right)(g)\in\Kalg.
    \]
Since $\lambda_n\neq 0$, we have proved that $f'$ is both differentially algebraic over $\bK$ and
belongs to $\Kalg$. We conclude thanks to Lemma~\ref{lemma:DiffAlg} that $f'\in\bK$.
\end{proof}


Now we consider the case $b=0$.
As we have already pointed out in the introduction, we need consider solutions of the functional equation
$\phir(y)=ay$ in a general $\Kalg$-algebra.
The reason for showing such a general statement is that Theorem~\ref{thmINTRO:main-0} provides a solution $f$ for $\phir(y)=ay+b$ by assumption,
but one has to construct some abstract auxiliary solutions of the associated homogeneous equation and recover the whole space of solution, to be able to study the nature of $f$ itself.

Thanks to \cite[Theorem, page 164]{wibmer_existence_2012} (see also \cite[Proposition~23]{bostan_differential_2020}), we know that for a given
functional equation $\phir(y)=ay$ there exists
a $\Kalg$-algebra $F$ such that:
\begin{quote}\mbox{}
\begin{minipage}{.1\textwidth}$(\cF)$\end{minipage}
\begin{minipage}{.7\textwidth}
\begin{enumerate}[itemsep=0em]
    \item $F$ comes equipped with an extension
        of $\phir$, such that $F^{\phir}=\C$, and an extension of $\frac{d}{dt}$, respecting the
        commutation
        $\frac{d}{dt}\circ\phir=R'(t)\phir\circ\frac{d}{dt}$;
    \item there exists a nonzero solution $z\in F$ of $\phir(y)=ay$ such that $F$ is generated over $\Kalg$ by $z,z^{-1}$ and all the derivatives of $z$;
    \item $F$ has no nilpotent elements and any element which is not invertible is a zero divisor.
\end{enumerate}
\end{minipage}
\end{quote}

In this case we have:

\begin{prop}\label{prop:maint-multiplicative-ALG}
Under the assumptions $(\cDT_{alg})$, let $a\in\bK$ and let $F$ be a $\Kalg$-algebra satisfying $(\cF)$.
Then either $z$ is differentially transcendental over $\C(t)$, with respect to $\frac{d}{dt}$, or
$\frac{z'}{z}$ is algebraic over $\C(t)$, i.e., $\frac{z'}{z}\in\bK$.
\end{prop}
We need to prove a generalization of Lemma~\ref{lemma:mainthm-multiplicative}, which has a very similar proof to the lemma below: 

\begin{lemma}\label{lemma:main-multiplicative-alg}
Under the assumptions of Proposition~\ref{prop:maint-multiplicative-ALG},
let $a\in\C$, $a\neq 0$, and $F$ be a $\Kalg$-algebra satisfying the assumption $(\cF)$.
Then either $a=1$ and $z\in\C$ or
$z$ is differentially transcendental over $\C(t)$ and there exists a nonzero $c\in\C$ such that
$\frac{z'}{z}=\frac{c}{\Psi}$.
\end{lemma}

\begin{proof}
For $a=1$ the statement is trivial, so let us suppose that $a\neq 1$.
Since $\phir(z)=az$, we also have $\phir(\Psi z')=a(\Psi z')$.
As
$F^{\phir}=\C$, there exists $c\in\C$ such that
$\Psi z'=cz$.
By contradiction, if $z$ is differentially algebraic over $\C(t)$,
then $\frac{z'}{z}=c\Psi^{-1}\in\K$ is also differentially algebraic over $\C(t)$.
By assumption, $\Psi$ is differentially transcendental over $\C(t)$, hence we deduce that $c=0$
and that $z$ is a constant,
in contradiction with the fact that $a\neq 1$.
We conclude that $z$ is differentially transcendental over $\C(t)$.
\end{proof}


\begin{proof}[Proof of Proposition~\ref{prop:maint-multiplicative-ALG}.]
If $a=1$ (resp. $a=0$), then $z\in\C\subset\C(t)$ (resp. $z=0\in\C(t)$).
Moreover, if $a\in\C\setminus\{0,1\}$ then $z$ is differentially transcendental
over $\C(t)$, by Lemma~\ref{lemma:main-multiplicative-alg}.
Therefore let us suppose that $a$ is not a constant and that $z$ is differentially algebraic over $\bK$
(or equivalently over $\C(t)$, since $\bK$ is the
algebraic closure of $\C(t)$)
and prove that $\frac{z'}{a}\in\bK$.
By taking the logarithmic derivative of $\phir(z)=az$ we find that
    \[
    \phir\l(\frac{\partial z}{z}\r)=\frac{\partial z}{z}+\frac{\partial a}{a},
    \hbox{~or equivalently that~}
    \phir\l(\frac{z'}{z}\r)R(t)'=\frac{z'}{z}+\frac{a'}{a}.
    \]
We point out that $\frac{a'}{a}\neq 0$, hence $\frac{z'}{z}\neq 0$.
The proof follows the structure of the proof of Theorem~\ref{thmINTRO:main-additive}, apart from the fact we cannot apply Proposition~\ref{prop:DiffAlgHanoi} since $F$ is not a field.
We will apply \cite[Proposition~28]{bostan_differential_2020} instead, which is a similar statement.

%
It follows from \cite[Proposition~28]{bostan_differential_2020}, 
that there exist an integer $n\geq 0$, $\lambda_0,\dots,\lambda_n\in\C$, not all zero, and $g\in\Kalg$ such that:
    \begin{equation}\label{eq:identity-z'/z-alg}
    \lambda_0\frac{\partial z}{z}+\dots
    +\lambda_n\partial^n\l(\frac{\partial z}{z}\r)=g\in\Kalg\,.
    \end{equation}
We can suppose that $\lambda_n\neq 0$.
If $n=0$ then $\frac{z'}{z}=\frac{g}{\lambda_0\Psi_0}$, therefore
$\frac{g}{\lambda_0\Psi_0}\in\Kalg$ is differentially algebraic over $\C(t)$.
Since $\Kalg/\bK$ is a purely differentially transcendental extension, we conclude that
$\frac{g}{\lambda_0\Psi_0}\in\bK$, and hence $\frac{z'}{z}\in\bK$.
\par
Let us consider the case $n\geq 1$. Let $L$ be the differentially algebraic extension of $\bK$ generated by
$z$ and all its derivatives.
Then $\frac{\partial z}{z}=\Psi_0\frac{z'}{z}\in\Psi_0\,L$.
One proves recursively that $\partial^n\l(\frac{\partial z}{z}\r)\in \frac{z'}{z}\Psi_0^n\Psi_n+L[\Psi_0,\dots,\Psi_{n-1}]$,
since on one hand we have:
    \[
    \partial\l(\frac{\partial z}{z}\r)=\Psi_0\l(\Psi_0\frac{z'}{z}\r)'=\frac{z'}{z}\Psi_0\Psi_1+\l(\frac{z'}{z}\r)'\Psi_0^2,
    \]
and on the other hand, as in Lemma~\ref{lemma:derivatives}, we have:
    \[
    \partial\l(\frac{z'}{z}\Psi_0^n\Psi_n\r)=\frac{z'}{z}\Psi_0^{n+1}\Psi_{n+1}+\frac{z'}{z}n\Psi_0^n\Psi_1\Psi_n
    +\l(\frac{z'}{z}\r)'\Psi_0^{n+1}\Psi_n\,.
    \]
Since $g\in\Kalg\cap(\frac{z'}{z}\Psi_0^n\Psi_n+L[\Psi_0,\dots,\Psi_{n-1}])$,
we conclude that $\frac{z'}{z}\in\bK$, applying
$\frac{1}{n!}\frac{\partial^n}{\partial\Psi_0^n}\frac{\partial}{\partial\Psi_n}$,
reasoning as we have done several time before.
\end{proof}


%
%

\subsection{A key result of parameterized Galois theory}
\label{subsec:GaloisTheory}

In this section, using the theory developed in \cite{hardouin_differential_2008}, to prove the main ingredient of Theorem~\ref{thmINTRO:main-0}.
This is the only section of the paper where we actively use some Galois theory.
For a short survey on of parameterized Galois theory the reader can also
see \cite{di_vizio_approche_2012}.

\par
Let $(C,\partial)$ be a differential closure of $\C$ (the latter being equipped with the trivial derivation).
This implies that (see, for instance \cite[Definition~1.7 and Theorem~2.2]{bouscaren_differentially_1998}):
\begin{enumerate}
    \item any differential equation with coefficients in $C$ that
    has a solution in a $C$-algebra, already has a solution in $C$;
    \item the field $C$ is a differentially algebraic extension of $\C$;
    \item the subfield of constants of $C$ is $\C$.
\end{enumerate}
For the reader which is not familiar with the notion of differential closure, the list of property above should be enough to
reason by analogy with the notion of algebraic closure, although the differential closure is a much more subtle objet and the analogy should be handled be some precautions.
The necessity of introducing such a huge and  mysterious field will become clear in a few lines.
\par
We want to extend the scalars from $\C$ to $C$, which we can do thanks to the following lemma:

\begin{lemma}
Let $L$ be one of the fields $\C(t)$, $\bK$, $\Kalg$, $\bF$ or $\Falg$.
Then $L\otimes_\C C$ is a domain.
\end{lemma}

\begin{proof}
See   \cite[\href{https://stacks.math.columbia.edu/tag/0FWF}{Tag 0FWF}]{the_stacks_project_authors_stacks_2023} or
the equivalent geometric formulation \cite[Proposition~5.51]{gortz_algebraic_2010}.
\end{proof}

For any choice of $L$ we can consider the field of fractions of $L\otimes_\C C$.
Extending intuitively the Notation~\ref{notation:argebraic-R}, we obtain the following fields:
    \[
    C\subset C(t) \subset \bK_C \subset \Kalg_{C} \subset \bF_C \subset \Falg_{C}.
    \]
For instance, $\Falg_{C}$ is the field of fraction of $\Falg\otimes_\C C$.
The map $\phir$ acts naturally on all such fields and the fields of $\phir$-invariant elements is $C$ for each one of them.
\change{Notice that $\bK_C$ is contained in the algebraic closure of $C(t)$ and $\bF_C$ is contained in the field of Puiseux series in $t$
with coefficients in $C$.}
\par
We have supposed that $\Psi$ is differentially transcendental over $\bK$. Since $C$ is differentially algebraic over $\C$,
we conclude that $\bK_C$ is differentially algebraic over $\bK$ and, hence, that
$\Psi$ is differentially transcendental over $\bK_C$, by Lemma~\ref{lemma:DiffAlg}.
\par
Let us consider a functional equation of the form
\begin{equation}\label{eq:system}
    \phir(\Vec{y})=
    \begin{pmatrix}
    a &b \\0 &1
    \end{pmatrix}\Vec{y},
\end{equation}
with $a,b\in\Kalg_{C}$.
There exists a $\Kalg_{C}$-algebra $\cR$ (see \cite[Definition~2.3]{hardouin_differential_2008}) such that:
\begin{enumerate}[itemsep=0em]
    \item $\cR$ comes equipped with an extension of $\phir$
    and an extension of $\partial$, respecting the commutation: $\phir\circ\partial =\partial\circ\phir$;
    \item there exists a nonzero solution
    $Z\in\GL_2(\cR)$ such that (each column of) $Z$ is a solution of \eqref{eq:system} and $\cR$ is generated over $\Kalg$, as a ring, by $Z$ and all its derivatives with respect to $\partial$ and $\det Z^{-1}$;
    \item $\cR$ has no nontrivial ideals which are
    invariant by both $\partial$ and $\phir$.
\end{enumerate}
Notice that $Z$ can be chosen of the form
$\begin{pmatrix}w & f\\0 &1\end{pmatrix}$,
with $\phir(w)=aw$ and $\phir(f)=af+b$. Moreover,
$\cR^{\phir}=C$ and $\cR$ is unique up to an isomorphism of
$\Kalg_{C}$-algebras, commuting with $\partial$ and $\phir$.
See \cite[Proposition~2.4]{hardouin_differential_2008}.
\par
The group
$G:=\hbox{Aut}^{\partial,\phir}(\cR,\Kalg_{C})$ of the automorphisms of
$\cR$ as a $\Kalg_C$-algebra, commuting to $\partial$ and $\phir$, is the
parameterized Galois group of \eqref{eq:system}.
It acts on $Z$ respecting the functional equations satisfied by its entries, therefore for any
$\phi\in G$ we have:
    \[
    \phi(Z)=Z\begin{pmatrix}c_\phi & d_\phi\\0&1\end{pmatrix},
    \]
with $c_\phi,d_\phi\in C$, so that $\phi(w)=c_\phi w$ and $\phi(f)=f+d_\phi w$.
An important property of $G$ is that
it can be identified to a subgroup of $\GL_2(C)$ ``defined by differential equations''.
Roughly, $c_\phi$ and $d_\phi$ are solutions of a set of differential equations that
are compatible with the multiplication of matrices. See \cite[Theorem~2.6]{hardouin_differential_2008}.
Such a description of the parameterized Galois group explains the need to extend
the constants to a differentially closed field: if $C$ is not differentially closed we are not sure
to be able to find $c_\phi$ and $d_\phi$, and therefore
an automorphism of $\cR$ over $\Kalg_{C}$, commuting with $\partial$ and $\phir$, and different
from the identity automorphism.
Its existence is a key point of the proof of Proposition~\ref{prop:key-Galois-theory} below.
\par
\change{Let $F$ be the total ring of fractions of $\cR$, so that any $\phi\in G$ extends to an isomorphism of $F$.
Then the elements of $F$ have following important properties: the $g\in F$ has the property that $\phi(g)=g$ for all $\phi\in G$ if and only if $g\in\Kalg_{C}$. See \cite[Theorem~2.7]{hardouin_differential_2008}.}
Two \change{other} (independent) facts play a crucial role in the proof below.
First of all, the algebra $F$ satisfies the third assumption of $(\cF)$.
Secondly, $\Psi_n\in\Kalg_{C}$ for any $n\geq 0$,
therefore $\phi(\Psi_n)=\Psi_n$, for any $\phi$ in $G$.
It follows that $\phi$ commutes not only with $\partial=\Psi_0\frac{d}{dt}$ and all its iterations,
but also with $\frac{d}{dt}$ and all its higher order derivatives.
Namely we have $\phi(g')=\phi(g)'$ for any $g\in F$.

We are now ready to prove the main statement of this subsection:

\begin{prop}\label{prop:key-Galois-theory}
Let us consider an equation of the form $\phir(y)=ay+b$, with $a,b\in\bK$, such that $a\neq 0,1$ and $b\neq 0$.
Let $f\in \Falg\setminus \Kalg$
satisfy the equation $\phir(f)=af+b$.
\change{If $f$ is differentially algebraic over $\bK$, then $\phir(y)=ay$ has a solution $z$ in a convenient $\bK$-algebra, such that
$\frac{z'}{z}\in\bK$.}
\end{prop}

\begin{proof}
Notice that $\bK\subset\bK_C$, $\Kalg\subset\Kalg_{C}$ and $\Falg\subset\Falg_{C}$.
By assumption, $f\in\Falg\setminus\Kalg$. In particular,
we know that $f$ does not belong to $\bK$ and we want to show that $f$ does not belong to $\bK_C$ either.
Let us suppose towards a contradiction that $f\in\bK_C$, i.e., that there exists a
polynomial $P$ in $2$ variables and with coefficients in $C$ such that
    \[
    P(t, f)=0.
    \]
\change{The polynomial $P$ cannot have all its coefficients in $\C$, because $f\not\in\bK$, by assumption. 
Moreover, the polynomial $P$ has a finite number of coefficients, 
therefore they generate a finite dimensional $\C$-vector space, contained in $C$. 
Let $p_1,\dots,p_r\in C$ be a $\C$-basis of such a vector space. Then we can rewrite 
$P$ as $P=\sum_{i=1}^rp_iP_i$, where each $P_i$ is a polynomial in $2$ variable with coefficients in $\C$, so that:
    \[
    0=P(t, f)\sum_{i=1}^rp_iP_i(t,f).
    \]
Once more we must have $P_i(t, f)\neq 0$, for any $i=1,\dots,r$.
Since $f$ is a Puiseux series with coefficients in $\C$,
any $P_i(t,f)$ is also a nonzero Puiseux series with coefficients in $\C$.
Therefore, we can develop the equation above according to the fractional powers of $t$ involved and find at least one nontrivial $\C$-linear combination 
of $p_1,\dots,p_r$ equal to $0$, which contradicts the construction of the $p_i$'s. 
So $f\not\in\bK_{C}$. Moreover, since $f$ is differentially algebraic over $\bK_C$, and
$\Kalg_{C}$ is purely differentially transcendental over $\bK_C$,
we also conclude that $f\not\in\Kalg_{C}$.}
\par
Now we are getting to the core of the proof.
By \cite[Proposition~6.17]{hardouin_differential_2008} or \cite[Theorem, page 164]{wibmer_existence_2012},
there exists an $\Falg_{C}$-algebra, satisfying $(\cF)$.
Since such a ring also contains $f$, applying  \cite[Proposition~6.17]{hardouin_differential_2008} again, one shows that it
contains a copy of the ring $\cR$ introduced above,
which is unique up to a isomorphism commuting with $\phir$ and $\partial$.
\par
Let $G$ be the Galois group of $\phir(y)=ay+b$ over $\Kalg_{C}$, that we have described above.
Since $f\not\in\Kalg_{C}$, there exists $\phi\in G$ such that $\phi(f)\neq f$.
By assumption, $f$ is differentially algebraic over $\bK$
with respect to $\frac{d}{dt}$,
hence there exists an integer $N\geq 0$ such
that the differential equation satisfied by $f$
can be written in the form:
    \[
    Q(f,f',\dots, f^{(N)})=0,
    \]
where $Q$ is a polynomial in $N+1$ indeterminates with coefficients in $\bK$.
Applying $\phi$ to the equation above we find that $\phi(f)$ is also differentially algebraic over $\bK$,
and actually satisfies the same differential
equation $Q(\phi(f),\phi(f)',\dots,\phi(f)^{(N)})=0$.
It follows that $z:=f-\phi(f)$, which is a nonzero solution of $\phir(y)=ay$, is differentially algebraic over
$\bK$. Therefore, we have proved that there exists a $\Kalg$-algebra satisfying $(\cF)$ and containing a
differentially algebraic solution $z$ of $\phir(y)=ay$.
It follows from Proposition~\ref{prop:maint-multiplicative-ALG} that $\frac{z'}{z}\in\bK$. This ends the proof.
\end{proof}

The following corollary is simply a rephrasing of Proposition~\ref{prop:key-Galois-theory}:

\begin{cor}\label{cor:key-Galois-theory}
We assume $(\cDT_{alg})$.
Let us consider an equation of the form $\phir(y)=ay+b$, with $a,b\in\bK$.
Let $f\in \Falg$
satisfy the equation $\phir(f)=af+b$.
If the functional equation $\phir(y)R'=y+\frac{a'}{a}$ does not have any algebraic solution, then either $f\in\bK$ or
$f$ is either differentially transcendental over $\bK$.
\end{cor}

\begin{proof}
If $a=0=b$, then $f=0\in\bK$. If $a=0$, but $b\neq 0$, then $\phir(f)=b$, hence $f=\phir^{-1}(b)\in\bK$.
Finally, if $b=0$, but $a\neq 0$, then $f$ is differentially transcendental by Proposition~\ref{prop:maint-multiplicative-ALG},
since $\frac{f'}{f}$ is a solution of $\phir(y)R'=y+\frac{a'}{a}$, which has no algebraic solution. So we can suppose that both $a$ and $b$ are nonzero.
\par
Let us suppose that $f$ is not algebraic and show that $f$ is necessarily differentially transcendental over $\bK$.
If, by contradiction, $f\notin\bK$ is differentially algebraic over $\bK$, then $f\not\in\Kalg$:
Indeed, $\Kalg/\bK$ is a purely differentially transcendental extension and it is enough to apply Lemma~\ref{lemma:DiffAlg}.
It follows from Proposition~\ref{prop:key-Galois-theory} that there must exist a solution of
$\phir(y)=ay$ whose logarithmic derivative is algebraic, in contradiction with the assumption that
$\phir(y)R'=y+\frac{a'}{a}$ does not have any algebraic solution.
Therefore $f$ is differentially transcendental over $\bK$ and we have proved the corollary.
\end{proof}

\subsection{End of proof of Theorem~\ref{thmINTRO:main-0}}

Applying Lemma~\ref{lemma:rat-2} both to $\phir(y)=ay+b$ and $\phir(y)R'=y+\frac{a'}{a}$,
we can restate
Corollary~\ref{cor:key-Galois-theory} as follows:

\begin{cor}\label{cor:key-Galois-theory-bis}
We assume that $R$ satisfies $(\cR)$.
Let us consider an equation of the form $\phir(y)=ay+b$, with $R,a,b\in\C(t)$ and $a,b\neq 0$.
Let $f\in \Falg$
satisfy the equation $\phir(f)=af+b$.
If the functional equation $\phir(y)R'=y+\frac{a'}{a}$ does not have any rational solution,
then either $f\in\C(t)$ or
$f$ is differentially transcendental over $\C(t)$.
\end{cor}

\begin{proof}
Because of the previous remark, the statement hold under the assumption $(\cDT)$. If $\Psi$ is differentially algebraic, this is the aforementioned statement
on Mahler equation.
\end{proof}

Notice that, thanks to Theorem~\ref{thmINTRO:b=0},
the corollary above is nothing else then Theorem~\ref{thmINTRO:main-R-algebraic}, mentioned in the introduction.
We are finally ready to complete the proof of Theorem~\ref{thmINTRO:main-0}.
We can work under the assumption $(\cDT)$, since the case where $\Psi$ is differentially algebraic has been already discussed at the beginning of this section.
We summarize the situations in which assuming $(\cDT)$  we have already proved Theorem~\ref{thmINTRO:main-0}:
\begin{enumerate}
  \item If $a=0$, the functional equation $\phir(y)=b$ has an algebraic solution, while if $b=0$, we have already proved the statement in
  Theorem~\ref{thmINTRO:main-multiplicative}.
  \item If all the nonzero solutions of $\phir(y)=ay$ are differentially transcendental over $\C(t)$,
  i.e.  $\phir(y)R'=y+\frac{a'}{a}$ does not have any rational solution, then
  we have prove the theorem in Corollary~\ref{cor:key-Galois-theory-bis}.
  \item If $\phir(y)=ay$ has a rational solution $z$, then we apply Theorem~\ref{thmINTRO:main-additive} to the solution $\frac{f}{z}$ of the functional equation
  $\phir(y)=y+\frac{b}{az}$. Therefore we have proved the theorem in this case, too. This covers in particular the case $a=1$.
\end{enumerate}
We need to show the statement under the assumption that $\phir(y)=ay$ has a nonzero differentially algebraic solution, which is not rational.
We have seen in
Theorem~\ref{thmINTRO:main-multiplicative} that in this case $z\in\bK$ and $\alpha:=\frac{z'}{z}\in\C(t)$.
The quotient $\frac{f}{z}$ is solution of the functional equation with algebraic coefficients $\phir(y)=y+\frac{b}{az}$, therefore
we know by Proposition~\ref{prop:main-additive-ALG} that
    \[
    \left(\frac{f}{z}\right)'=\frac{f'-\alpha f}{z}\in\bK.
    \]
We set $\beta:=f'-\alpha f$, which belongs to $\bK$, because $z$ does. Since
    \[
    \phir\left(\Big(\frac{f}{z}\Big)'\right)R'=\left(\frac{f}{z}\right)'+\frac{1}{z}\left[\left(\frac{b}{a}\right)'-\alpha\frac{b}{a}\right],
    \]
one obtain by a direct calculation that $\phir(\beta)=\frac{a}{R'}\beta+a\left[\left(\frac{b}{a}\right)'-\alpha\frac{b}{a}\right]$.
If $\left(\frac{b}{a}\right)'=\alpha\frac{b}{a}$, then $z$ and $\frac{b}{a}$ coincide up to a multiplicative constant, hence $z\in\C(t)$, against the assumption.
Therefore $a\left[\left(\frac{b}{a}\right)'-\alpha\frac{b}{a}\right]\neq 0$. Notice that $\frac{a}{R'}\neq 0$, by assumption.
By Lemma~\ref{lemma:rat-2}, if we show that the functional equation $\phir(y)=\frac{a}{R'}y$ does not have any nonzero algebraic solution,
we can conclude that $\beta\in\C(t)$, which would complete the proof of the theorem.
Let us suppose by contradiction that there exists $w\in\bK$, $w\neq 0$, such that $\phir (w)=\frac{a}{R'}w$.
Since $\phir(\Psi)=R'\Psi$, we conclude that $\phir(w\Psi)=aw\Psi$. Therefore there exists $c\in\C$ such that $z=cw\Psi$, which
implies that $z$ is differentially trascendental, against the assumption. We have obtained a contradiction, hence $\phir(y)=\frac{a}{R'}y$ does not have any nonzero algebraic solution.


\appendix


\noindent{{\sc L. Di Vizio}\\
CNRS et UVSQ, Laboratoire de Mathématiques de Versailles,
45 Avenue des États-Unis, 78035 Versailles, FRANCE. \href{mailto:lucia.di.vizio@math.cnrs.fr}{lucia.di.vizio@math.cnrs.fr}} ORCID 0000-0003-4313-9072\\

\noindent{{\sc G. Fernandes}\\UVSQ, Laboratoire de Mathématiques de Versailles,
45 Avenue des États-Unis, 78035 Versailles, FRANCE.
\href{mailto:gwladys.fernandes.maths@gmail.com}{gwladys.fernandes.maths@gmail.com}} ORCID 0000-0003-0832-6614\\

\noindent{{\sc M. Mishna}\\ Department of Mathematics, Simon Fraser University, Burnaby BC, V5A 1S6 CANADA.   \href{mailto:mmishna@sfu.ca}{mmishna@sfu.ca}} ORCID 0000-0001-5197-5973\\
 
\end{document}